\documentclass[12pt,reqno]{amsart}
\usepackage{latexsym}
\usepackage{amssymb}
\usepackage{mathrsfs}
\usepackage{amsmath}
\usepackage{comment}
\usepackage{fancybox,color}
\usepackage{enumerate}
\usepackage[latin1]{inputenc}
\usepackage{soul} 
\usepackage[colorlinks=true, linkcolor=blue, citecolor=blue]{hyperref}
\usepackage{graphicx}
\usepackage{autonum}
\allowdisplaybreaks[4]
{\catcode`p =12 \catcode`t =12 \gdef\eeaa#1pt{#1}}      
\def\accentadjtext#1{\setbox0\hbox{$#1$}\kern   
	\expandafter\eeaa\the\fontdimen1\textfont1 \ht0 }
\def\accentadjscript#1{\setbox0\hbox{$#1$}\kern 
	\expandafter\eeaa\the\fontdimen1\scriptfont1 \ht0 }
\def\accentadjscriptscript#1{\setbox0\hbox{$#1$}\kern   
	\expandafter\eeaa\the\fontdimen1\scriptscriptfont1 \ht0 }
\def\accentadjtextback#1{\setbox0\hbox{$#1$}\kern       
	-\expandafter\eeaa\the\fontdimen1\textfont1 \ht0 }
\def\accentadjscriptback#1{\setbox0\hbox{$#1$}\kern     
	-\expandafter\eeaa\the\fontdimen1\scriptfont1 \ht0 }
\def\accentadjscriptscriptback#1{\setbox0\hbox{$#1$}\kern 
	-\expandafter\eeaa\the\fontdimen1\scriptscriptfont1 \ht0 }

\newcommand{\ch}[1]{{\mbox{\raise 1pt\hbox{\large$\chi$}}}_{\lower 1pt\hbox{$\scriptstyle #1$}}}

\usepackage[left=2.1cm,top=2.3cm,right=2.1cm]{geometry}

\def\1{\raisebox{2pt}{\rm{$\chi$}}}

\newtheorem{theorem}{Theorem}[section]

\newtheorem{corollary}[theorem]{Corollary}
\newtheorem{lemma}[theorem]{Lemma}
\newtheorem{proposition}[theorem]{Proposition}

\theoremstyle{definition}
\newtheorem{definition}[theorem]{Definition}
\newtheorem{remark}[theorem]{Remark}

\newcommand{\R}{{\mathbb R}}

\newcommand{\N}{{\mathbb N}}
\newcommand{\Z}{{\mathbb Z}}

\newcommand\diam{\operatorname{diam}}

\def\1{\raisebox{2pt}{\rm{$\chi$}}}

\def\vint_#1{\mathchoice%
	{\mathop{\kern 0.2em\vrule width 0.6em height 0.69678ex depth -0.58065ex
			\kern -0.8em \intop}\nolimits_{\kern -0.4em#1}}%
	{\mathop{\kern 0.1em\vrule width 0.5em height 0.69678ex depth -0.60387ex
			\kern -0.6em \intop}\nolimits_{#1}}%
	{\mathop{\kern 0.1em\vrule width 0.5em height 0.69678ex
			depth -0.60387ex
			\kern -0.6em \intop}\nolimits_{#1}}%
	{\mathop{\kern 0.1em\vrule width 0.5em height 0.69678ex depth -0.60387ex
			\kern -0.6em \intop}\nolimits_{#1}}}
\def\vintslides_#1{\mathchoice%
	{\mathop{\kern 0.1em\vrule width 0.5em height 0.697ex depth -0.581ex
			\kern -0.6em \intop}\nolimits_{\kern -0.4em#1}}%
	{\mathop{\kern 0.1em\vrule width 0.3em height 0.697ex depth -0.604ex
			\kern -0.4em \intop}\nolimits_{#1}}%
	{\mathop{\kern 0.1em\vrule width 0.3em height 0.697ex depth -0.604ex
			\kern -0.4em \intop}\nolimits_{#1}}%
	{\mathop{\kern 0.1em\vrule width 0.3em height 0.697ex depth -0.604ex
			\kern -0.4em \intop}\nolimits_{#1}}}

\newcommand{\intav}{\vint}

\newcommand{\Om}{\Omega}

\title[Limit at Infinity]{Limits at infinity for functions in fractional Sobolev spaces}

\author[A. Agarwal]{Angha Agarwal}   %
\address[A.I.]{Department of Mathematics and Statistics, P.O. Box 35, FI-40014 University of Jyvaskyla, Finland}
\email{angha.a.agarwal@jyu.fi}

\author[P. Koskela]{Pekka Koskela}   %
\address[P.K.]{Department of Mathematics and Statistics, P.O. Box 35, FI-40014 University of Jyvaskyla, Finland}
\email{pekka.j.koskela@jyu.fi}

\author[K. Mohanta]{Kaushik Mohanta}   %
\address[K.M.]{Department of Mathematics and Statistics, P.O. Box 35, FI-40014 University of Jyvaskyla, Finland}
\email{kaushik.k.mohanta@jyu.fi}



\begin{document}
	
	\begin{abstract}
		We establish optimal results on limits at infinity for functions in fractional Sobolev spaces.
	\end{abstract}

	\maketitle

	\section{Introduction}
	This article aims to study the `pointwise' behavior of functions in homogeneous fractional Sobolev spaces, near infinity. First, we need to be clear about what ``limit at infinity'' means, because elements of $L^p(\R^n)$ are not functions, rather than equivalence classes of functions. We start by fixing $\R^n$ as the Euclidean space with dimension $n\geq2$. It is known that every $u\in W^{1,p}(\R^n)$ has a representative $\tilde{u}$ such that
	$$
	\lim_{t\to\infty} \tilde{u}(t\xi) =0, 
	$$
	for $\mathcal{H}^{n-1}$-a.e. $\xi\in \mathbb{S}^{n-1}$ for any $1\leq p<\infty$, and, for $p>n$, $u(x)\to 0$ uniformly as $x\to\infty$. The situation changes when we consider homogeneous Sobolev spaces. In this case, there is no reason for the supposed limiting value to be zero. Uspenski\u{\i} and Timan (see \cite{Upenskii} and \cite{Timan}) showed that $1\leq p < n$ if and only if for every $u \in  \dot{W}^{1,p}(\R^n)$ there is some $c\in\R$ and a representative $\tilde{u}$ of $u$ such that
	\begin{equation} \label{uspenki limit}
		\lim_{t\to\infty} \tilde{u}(t\xi) =c
	\end{equation}
	for a.e. $\xi\in \mathbb{S}^{n-1}$. We call the limit $c$ in \eqref{uspenki limit} the unique almost sure radial limit at infinity of the function $u$. Next, a function $u:\R^n\to\R$ is said to have a unique almost sure vertical limit $c$ at infinity if for almost every $\overline{x}\in\R^{n-1}$ we have 
	\begin{equation}\label{fefferman limit}
		\lim_{t\to\infty}\tilde{u}(\overline{x},t)=c.
	\end{equation}
	The existence of such limits has been studied by Fefferman \cite{Fefferman} and Portnov \cite{Portnov}. They independently showed the existence of a unique almost sure vertical limit $c$ at infinity for functions in $\dot{W}^{1,p}(\R^n)$ under the assumption $1\leq p<n$. Moreover, it turns out that the values of the limits in \eqref{uspenki limit} and \eqref{fefferman limit} are the same. 
	
	The fractional Sobolev spaces $W^{s,p}(\R^n)$, defined in Section~\ref{preli}, are regarded as the fractional counterparts of the Sobolev spaces $W^{1,p}(\R^n)$. Both these spaces are special cases of the Triebel-Lizorkin spaces $F^{s}_{p,q}$: $F^{s}_{p,p}(\R^n)=W^{s,p}(\R^n)$ for $s\in (0,1)$ and  $F^{1}_{p,2}(\R^n)=W^{1,p}(\R^n)$. From Theorem~1.2 of \cite{PrEe}, and also from \cite{Triebel}, we know that, for $p,q\in [1,\infty)$, $s>\frac{n}{p}-\frac{n}{q}$, we have 
	$$
	F^{s}_{p,q}(\R^n)=\left\{ u\in L^{\max\{p,q\}}(\R^n) \cap L^p(\R^n) \mid [u]_{W^{s,p}_q(\R^n)}<\infty \right\},
	$$
	where $[u]_{W^{s,p}_q(\R^n)}:=\left\|\left\|\frac{u(x)-u(y)}{|x-y|^{\frac{n}{q}+s}}  \right\|_{L^q(dy)} \right\|_{L^p(dx)}$. We define the homogeneous spaces $\dot{W}^{s,p}_q(\R^n)$ via this quasinorm (see Section~\ref{preli} for details). These spaces may exhibit significantly different behaviour when $s<\frac{n}{p}-\frac{n}{q}$; for example, in this range of parameters, $C_c^\infty$ is not contained in $\dot{W}^{s,p}_q(\R^n)$, see Remark~3.8 of \cite{PrEe}. To the best of our knowledge, the limits of functions $u\in \dot{W}^{s,p}_q(\R^n)$, in any pointwise sense, has not yet been studied even in the case $p=q$.
	
	Our aim in this paper is to study vertical and radial limits at infinity, in a generalized sense, for functions in the space $\dot{W}^{s,p}_q(\R^n)$. The method of the proof for the classical results \cite{Fefferman,Portnov} uses the fact that $W^{1,p}(I)\subseteq W^{1,1}(I)$ when  $p\geq1$ and $I$ is a bounded interval in $\R$. Due to the non-local nature of the $W^{s,p}_q$-seminorms, the methods used there cannot be used for our purpose. For instance, we do not have the embedding $W^{s,p}(I)\subseteq W^{s,1}(I)$, see \cite{MiSi}.


	We now proceed to state our main results. We will take $\dot{W}^{s,p}_q(\R^n)$ as the collection of all precise representatives $u^*$ of measurable functions $u$ of finite $W^{s,p}_q$-seminorm. The precise representative is defined in Definition~\ref{goog-rep}. Using this, we state our first main result showing the existence of a limit outside a ``thin" set. The measure of the thinness at infinity in dimension depending on $s$ and $p$ is referred to as $(p,s)$-thin, defined in Definition~\ref{ps-thin}. 
	\begin{theorem}\label{ps-thin theorem}
		Let $u^*\in\dot{W}^{s,p}_q(\R^n)$ with $0<s<1$, $0<p,q<\infty$ and $sp<n$. Then there exists a set $E$ which is $(p,s)$-thin at infinity and $K_u\in\R$ so that 
		\begin{equation}
			\lim_{x\to\infty, x\notin E}u^*(x)=K_{u^*}.
		\end{equation}
		
	\end{theorem}
	Let us single out two consequences.  
	\begin{corollary}\label{rad coro}
		Let $0<s<1$, $0<p,q<\infty$, $1\leq k<sp<n$, and $u^*\in\dot{W}^{s,p}_q(\R^n)$. Then, for each $(n-k)$-dimensional sphere $\mathbb{S}^{n-k}$ centred at origin of $\R^n$, $u^*$ has limit $K_{u^*}\in \R$ at infinity along $\mathcal{H}^{n-k}$-a.e. $k$-plane perpendicular to $\mathbb{S}^{n-k}$.
	\end{corollary}
	
	\begin{corollary}\label{ver coro}
		Let $0<s<1$, $0<p,q<\infty$, $1\leq k<sp<n$. Let $u^*\in\dot{W}^{s,p}_q(\R^n)$. Then, for each $(n-k)$-dimensional subspace $V^{n-k}\subset \R^n$, $u^*$ has limit $K_{u^*}\in\R$ at infinity along $\mathcal{H}^{n-k}$-a.e. $k$-plane perpendicular to $V^{n-k}$.
	\end{corollary}
	In the above two statements, the term ``limit at infinity'' refers to the existence of limit at infinity along the $k$-plane in the usual sense. Especially, this limit exists along $\mathcal{H}^{n-k}$-every $k$-plane perpendicular to $\R^{n-k}$. By choosing $k=1$, we obtain the existence of vertical limits. Similarly, Corollary~\ref{rad coro} gives the existence of unique almost sure radial limits. The precise definitions are given in Section~\ref{preli} (see Definitions~\ref{a.e.limit} and \ref{limit-at-infty}). 
	
	Theorem~\ref{ps-thin theorem} also has a somewhat unexpected application: the space $\dot{W}^{s,p}_q(\R^n)$ only contains constant functions when $0<s<1$, $0< p<\infty$ and $\frac{np}{n-sp}\leq q<\infty$, compare with \cite[Remark~3.8]{PrEe}. Notice that $q\geq\frac{np}{n-sp}$ is equivalent to $s\leq\frac{n}{p}-\frac{n}{q}$.
	
	\begin{corollary}\label{only constant function}
		Let $u\in\dot{W}^{s,p}_q(\R^n)$ with $0<s<1$, $0< p<\infty$, $q\geq \frac{np}{n-sp}$ and $sp<n$. Then there is $K_{u^*}\in\R$ so that $u(x)=K_{u^*}$ for almost every $x\in\R^n$.
	\end{corollary}
	Our technique also allows us to deal with the spaces $W^{s,p}_q(\R^n)=\dot{W}^{s,p}_q(\R^n)\cap L^p(\R^n)$.

	\begin{corollary}\label{Lp convg 0}
		Let $0<s<1$, $0<p,q<\infty$ and $u^*\in {W}^{s,p}_q(\R^n)$. Then there exists a set $E$ which is $(p,s)$-thin at infinity so that $u^*(x)\to 0$ when $|x|\to\infty$ in $\R^n\setminus E$.
	\end{corollary}
	
	One cannot deduce Theorem~\ref{ps-thin theorem} from Corollary~\ref{Lp convg 0} for $0<q<\frac{np}{n-sp}$ since there exist functions in $\dot{W}^{s,p}_q(\R^n)$ for which $u-c\in L^p(\R^n)$ for no $c\in\R$, see Proposition~\ref{u-c for no c} below. Also see \cite{GiKaSh} for related studies. 
	\begin{proposition}\label{u-c for no c}
		Let $0<s<1$, $0<p<\infty$ and $q\in(0,\frac{np}{n-sp})$. Then there exists a function  $u\in\dot{W}_q^{s,p}(\R^n)$ for which $u-c\in L^p(\R^n)$ for no $c\in\R$.
	\end{proposition}
	Let us discuss the case $sp\notin (1,n)$, not covered by Corollary~\ref{rad coro} and Corollary~\ref{ver coro}. We begin with the case $sp\geq n$.

	
	\begin{proposition}\label{sp geq n lemma}
		Let  $0<s<1$, $0<p,q<\infty$ with $sp\geq n$. Then there exists a continuous function $u$ in $\dot{W}^{s,p}_q(\R^n)$ such that for every curve $\gamma:[0,\infty)\to\R^n$ with $\gamma(0)=0$, $\lim_{t\to\infty}\gamma(t)=\infty$ we have 
		$$
		\lim_{t\to\infty}u(\gamma(t))=\infty.
		$$
	\end{proposition}

	\begin{proposition}\label{sp>n lemma}
		Let $0<p,q<\infty$, $0<s<1$ with $sp\geq n$. Then there exists a  bounded continuous function $u$ in $ \dot{W}^{s,p}_q(\R^n)$ such that there exists no curve $\gamma:[0,\infty)\to\R^n$ with $\gamma(0)=0$, $\lim_{t\to\infty}\gamma(t)=\infty$, for which the limit $\lim_{t\to\infty}u(\gamma(t))$ exists.
		
	\end{proposition}
	
	

	We have the following result for $sp<1$.

	\begin{proposition}\label{sp leq 1 lemma}
		Let $0<s<1$, $0<p,q<\infty$ with $sp<1$. Then there exists a bounded continuous function $u$ in $\dot{W}^{s,p}_q(\R^n)$ such that there exists no curve $\gamma:[0,\infty)\to\R^n$ with $\gamma(0)=0$, $\lim_{t\to\infty}\gamma(t)=\infty$, for which the limit $\lim_{t\to\infty}u(\gamma(t))$  exists.
	\end{proposition}   
	We are left with the case $sp=1$. In the case of the usual Sobolev spaces $W^{1,1}(\R^n)$, we have the existence of both radial and vertical limits. We do not have counterexamples for all values of $q$ when $sp=1$, but the following example covers the case $q\geq 1$.
	\begin{proposition}\label{sp=1}
		Let $0<s<1$, $0< p<\infty$, $1\leq q<\frac{np}{n-sp}$ with $sp=1$. Then there exists a bounded continuous function $u$ in $\dot{W}^{s,p}_q(\R^n)$ such that there exists no curve $\gamma:[0,\infty)\to\R^n$ with $\gamma(0)=0$, $\lim_{t\to\infty}\gamma(t)=\infty$, for which the limit $\lim_{t\to\infty}u(\gamma(t))$  exists.
	\end{proposition}
The functions referred to in the propositions above do not have complicated definitions but checking the membership in the respective fractional Sobolev Space is tedious. This is caused by the non-locality of our norms.


	Before proceeding further, we briefly outline the structure of the article. In Section~\ref{preli}, we give necessary definitions, explain our assumptions and notations used in the article, and collect some known results essential in the sequel. In Section~\ref{main}, we prove our central result, Theorem~\ref{ps-thin theorem}. The other main results for the existence of unique radial and vertical limits, i.e. Theorem~\ref{rad coro} and Theorem~\ref{ver coro} are proved in Section~\ref{radial section} and Section~\ref{vertical section}, respectively. The counterexamples mentioned above are constructed in Section~\ref{counter examples}. 
	\bigskip	
	
	\subsection*{Acknowledgement}
	The authors have been supported by the Academy of Finland via Centre of Excellence in Randomness and Structures Research (Project numbers 346310, 364210).
	\bigskip
	
	\section{Preliminaries}\label{preli}

	By a \textit{cube} in $\R^n$ we shall understand a set of the form $[a_1,b_1)\times \cdots \times [a_n,b_n)$ for real numbers $a_i<b_i$ for $i=1,\cdots, n$. The standard Lebesgue measure in $\R^n$ is denoted by $|\cdot|$. The $\lambda$-dimensional Hausdorff content is denoted by $\mathcal{H}_\infty^\lambda$. The corresponding Hausdorff measure is denoted by $\mathcal{H}^\lambda$.  For the convenience of the reader, we recall their definitions below. Let $\lambda\in (0,\infty)$ and $\delta\in(0,\infty]$. Given a set $A\subset\R^n$, a $\delta$-cover of $A$ is any countable family of sets $\{B_j\}_{j\in\N}$ so that 
	\begin{equation}
		A\subset \cup_{j\geq1}B_j\text{    and    }\diam(B_j)\leq\delta 
	\end{equation}
	for all $j\in\N$. The $\lambda$-dimensional Hausdorff $\delta$-content is 
	\begin{equation}
		\mathcal{H}_{\delta}^{\lambda}(A):=\inf \left\{\sum_{j\geq1}\diam(B_j)^{\lambda} \mid   A\subset \cup_{j\geq1}B_j,\text{  }\diam(B_j)\leq\delta \right\}.
	\end{equation}
	The $\lambda$-dimensional Hausdorff measure is defined as 
	\begin{equation}
		\mathcal{H}^{\lambda}(A):=\lim_{\delta\to0} \mathcal{H}_{\delta}^{\lambda}(A).
	\end{equation}
	We record some well-known properties of Hausdorff content below for future reference.
	\begin{lemma}[\cite{Evans18}]\label{lm-hausdorff-content}
		Let $\lambda\in (0,\infty)$ and $\delta\in(0,\infty]$. Let $A\subset\R^n$. Then\\ 
		
		\textbf{(i)}  $\mathcal{H}_{\infty}^{\lambda}(A)=0$ if and only if $\mathcal{H}^{\lambda}(A)=0$.\\
		
		\textbf{(ii)} Let $L:\R^n\to\R^m$ be Lipschitz continuous and $\delta>0$. Then 
		$$
		\mathcal{H}^{\lambda}_{Lip(L)\delta}(L(A))\leq (Lip(L))^{\lambda}\mathcal{H}^{\lambda}_{\delta}(A).
		$$
	\end{lemma}

	When $p\in [1,\infty)$ and $\Om\subset \R^n$ is an open set, we define the homogeneous Sobolev space $\dot{W}^{1,p}(\Om)$ to be the collection of all locally integrable functions $u:\Om\to \R$ whose weak gradients $\nabla u$ belong to $L^p(\Om,\R^n)$. We consider this vector space to be equipped with the seminorm. 
	$$
	[u]_{W^{1,p}(\Om)}:=\left(\int_{\Om}|\nabla u(x)|^pdx\right)^\frac{1}{p}.
	$$
	For $s\in (0,1)$ and $p,q\in (0,\infty)$, and an open set $\Om\subset \R^n$, we write $\dot{W}_q^{s,p}(\R^n)$ to denote the collection of all measurable functions $u:\Om\to\R$ for which
	\begin{equation}
		\int_{\Om}\left(\int_{\Om}\frac{|u(x)-u(y)|^q}{|x-y|^{n+sq}}dy\right)^{\frac{p}{q}}dx<\infty .
	\end{equation}
	The quantity 
	$$
	[u]_{W_q^{s,p}(\Om)} := \left(\int_{\Om}\left(\int_{\Om}\frac{|u(x)-u(y)|^q}{|x-y|^{n+sq}}dy\right)^{\frac{p}{q}}dx\right)^\frac{1}{p}
	$$
	gives a (quasi) seminorm for the above-mentioned range of $p,q,s$. By quasinorm we mean that for $u,v\in\dot{W}_q^{s,p}(\R^n)$ we have 
	\begin{equation}\label{quasinorm constant}
		[u+v]_{W_q^{s,p}(\Om)}\leq C ([u]_{W_q^{s,p}(\Om)}+[v]_{W_q^{s,p}(\Om)}) 
	\end{equation}  
	where $C>1$, depending on $p,q$ when either or both $p,q<1$, and $C=1$ when $p,q\geq1$. In the case $p=q$, we denote $\dot{W}^{s,p}_p(\Om)$ by $\dot{W}^{s,p}(\Om)$. The corresponding non-homogeneous spaces are $W^{1,p}(\Om):=\dot{W}^{1,p}(\Om)\cap L^p(\Om)$, $W^{s,p}(\Om):=\dot{W}^{s,p}(\Om)\cap L^p(\Om)$, and $W^{s,p}_q(\Om):=\dot{W}^{s,p}_q(\Om)\cap L^p(\Om)$. The norm of $u\in W^{s,p}_q(\Om)$ is
	\begin{equation}
		\|u\|_{W^{s,p}_q(\Om)}=\|u\|_{L^p(\Om)}+[u]_{W_q^{s,p}(\Om)}. 
	\end{equation}
	For an open set $A\subset \R^n$, the $(s,p,q)$-Capacity of $A$ for $0<s<1$, $1<p<\infty$ and $1<q<\infty$  is defined to be
	\begin{equation}
		\inf\{\|u\|_{W^{s,p}_q}^p\mid u\in W^{s,p}_q(\R^n), \ u\geq1 \text{ on a neighbourhood of $A$}\}
	\end{equation}
	and it is denoted by $\mathrm{Cap}_{s,p,q}(A)$. In the case $p=q$, we simplify notation by writing it as $\mathrm{Cap}_{s,p}(A)$.
	For future reference, we mention the following results:
	\begin{lemma}\cite[Theorem~5.1.9]{AdDa}\label{haus_cap_AdDa}
		Let $s\in (0,1)$, $p>1$, $sp<n$, and let $A\subset\R^n$. Then $\mathrm{C}_{s,p}(A)=0$, whenever $\mathcal{H}^{n-sp}(A)<\infty$. Here, $\mathrm{C}_{s,p}(A)$ denotes the Bessel capacity of the set $A$.
	\end{lemma}
	\begin{lemma}\cite[Proposition~4.4.4]{AdDa}\label{capacity comparison}
		Let $1<p<\infty$ and $0<p\leq n/s$. Then there is a constant $M$ such that for all $E\subset\R^n$ and $q\in[1,\frac{np}{n-sp})$
		$$M^{-1}\mathrm{C}_{s,p}(E)\leq \mathrm{Cap}_{s,p,q}(E)\leq M \mathrm{C}_{s,p}(E).$$
	\end{lemma}

			For any measurable function $u:\R^n\to\R$ with $u$ finite a.e. and a measurable subset $E\subset\R^n$ of finite positive measure, \textit{the median} of $u$, over the set $E$, is
			$$
			m_u(E):=\sup \left\{ M\in\R \ | \ |\{x\in E\ |\ u(x)< M\}| \geq \frac{|E|}{2}\right\}.
			$$
			For a fixed set $E$, we define
			\begin{equation}\label{partition on med}
				E^-:= \left\{ x\in E\ |\ u(x)\leq m_u(E) \right\}
				\quad 
				\text{and}
				\quad
				E^+:= \left\{ x\in E\ |\ u(x)\geq m_u(E) \right\}.
			\end{equation}
			Note that $2|E^+|\geq |E|$ and $2|E^-|\geq |E|$. We record two observations, which we shall use in the sequel.
			\begin{lemma}\label{l2}
				Let $E, F$ be sets of finite positive measure in $\R^n$. Let $u$ be a real-valued measurable function a.e. finite on $\R^n$. Then there exist $\tilde{E}\subset E, \tilde{F}\subset F$ such that $|\tilde{E}|\geq\frac{|E|}{2},|\tilde{F}|\geq\frac{|F|}{2}$ and 
				$$|m_u(E)-m_u(F)|\leq|u(x)-u(y)|$$for all $x\in\tilde{E}, y\in\tilde{F}$.
			\end{lemma}
			\begin{proof}
				Let $E, F$ be two subsets of finite measure in $\mathbb{R}^n$. If $$m_u(E)\leq m_u(F),$$ set $\tilde{E}=E^-$ and $\tilde{F}=F^+$.
				Then for all $x\in\tilde{E}$ and $y\in\tilde{F}$ we get 
				\begin{align}
					\label{1.1.1}
					0 \leq m_u(F)-  m_u(E)
					\leq u(y)-u(x)
					= |u(x)-u(y)|.
				\end{align}
				On the other hand if $$m_u(F)\leq m_u(E),$$ set $\tilde{E}=E^+$ and $\tilde{F}=F^-$. Then for all $x\in\tilde{E}$ and $y\in\tilde{F}$
				\begin{align}
					\label{1.1.2}
					0 \leq m_u(E)-  m_u(F)|
					= u(x)-u(y)
					\leq |u(x)-u(y)|.
				\end{align}
				The claim follows from inequalities $\eqref{1.1.1}$ and $\eqref{1.1.2}$.\end{proof}
			\begin{lemma}\label{med_another lemma}
				Let $E$ be a set of finite positive measure in $\R^n$. Let $u$ be a real-valued measurable function a.e. finite on $\R^n$. Let $c\in\R$. Then there exists a set $\tilde{E}\subset E$ such that $|\tilde{E}|\geq\frac{|E|}{2}$ and 
				$$|m_u(E)-c|\leq|u(x)-c|$$for all $x\in\tilde{E}$.
			\end{lemma}
			\begin{proof}
				Fix $c\in\R$. Let $E\subset \R^n$ be a set of finite positive measure. Suppose, $m_u(E)\geq c$. Choose $x\in E^{+}$ where $E^{+}$ is as in \eqref{partition on med}, then
				\begin{equation}
					0\leq m_u(E)-c \leq u(x) -c\leq |u(x)-c|.
				\end{equation}
				If $m_u(E)\leq c$ then for $x\in E^-$  where $E^{-}$ is as in \eqref{partition on med}, we have 
				\begin{equation}
					0\leq c-m_u(E)\leq c-u(x)\leq |u(x)-c|.
				\end{equation} Therefore we get that 
				\begin{equation}
					|m_u(E)-c|\leq |u(x)-c|
				\end{equation}
				for a set $\tilde{E}\subset E$ with $|\tilde{E}|\geq \frac{|E|}{2}$. 
			\end{proof}

			\begin{lemma}\label{l1}
				Let $0<p,q<\infty$. Let $E, F$ be subsets of finite measure in $\R^n$ and let $u$ be a real-valued measurable function a.e. finite on $\R^n$. Then we have
				\begin{align}\label{2.1.1}
					|m_u(E)-m_u(F)|^p\leq 2^{1+\frac{p}{q}}\intav_E\left(\intav_F|u(x)-u(y)|^qdy\right)^{\frac{p}{q}}dx.
				\end{align}
			\end{lemma}
			
			\begin{proof}
				Using Lemma \ref{l2} we find subsets $\tilde{E}\subset E, \tilde{F}\subset F$ of finite measure such that $|\tilde{E}|\geq\frac{|E|}{2},|\tilde{F}|\geq\frac{|F|}{2}$ and 
				$$|m_u(E)-m_u(F)|\leq|u(x)-u(y)|$$
				for all $x\in\tilde{E}, y\in\tilde{F}$. Then
				$$ |m_u(E)-m_u(F)|^q
				\leq\intav_{\tilde{F}}|u(x)-u(y)|^qdy.$$
				Hence, it follows that
				\begin{align}
					|m_u(E)-m_u(F)|^p&
					\leq\left(\intav_{\tilde{F}}|u(x)-u(y)|^qdy\right)^\frac{p}{q}\\
					&\leq\intav_{\tilde{E}}\left(\intav_{\tilde{F}}|u(x)-u(y)|^qdy\right)^\frac{p}{q}dx\\
					&\leq\frac{1}{|\tilde{E}|}\int_{\tilde{E}}\left(\frac{1}{|\tilde{F}|}\int_{\tilde{F}}|u(x)-u(y)|^qdy\right)^\frac{p}{q}dx\\
					&\leq\frac{2}{|E|}\int_{E}\left(\frac{2}{|F|}\int_{F}|u(x)-u(y)|^qdy\right)^\frac{p}{q}dx\\
					&\leq2^{1+\frac{p}{q}}\intav_{E}\left(\intav_{F}|u(x)-u(y)|^qdy\right)^\frac{p}{q}dx.
				\end{align}
			\end{proof}

			\begin{definition}\label{goog-rep}
				Let $0<s<1$, $0<p,q<\infty$ and $u\in\dot{W}^{s,p}_q(\R^n)$. We define the \textit{precise representative} $u^*$ of $u$  to be the function given by
				\begin{equation}
					u^*(x):=\limsup_{\underset{x\in Q}{|Q|\to0}}m_u(Q),
				\end{equation}
				for all $x\in \R^n$, where $Q$ is required to be a cube in $\R^n$. Note that $u^*(x)=u(x)$ a.e. in $\R^n$ by \cite[Lemma~2.2]{Fujii}. Hence $u^*\in\dot{W}^{s,p}_q(\R^n)$. With slight abuse of notation, we will assume that each $u\in\dot{W}^{s,p}_q(\R^n)$ is itself a precise representative in what follows. 
			\end{definition}

			Now let us formally define the term ``\textit{limit at infinity}'' of a function $u\in W^{s,p}_q(\R^n)$.
			\begin{definition}[Limit at infinity]\label{limit-at-infty}
				Let $P^k\subset \R^n$ be a $k$-dimensional affine subspace of $\R^n$ and $u\in \dot{W}^{s,p}_q(\R^n)$. We say that $u$ has a limit $c\in \R$ at infinity along $P^k$, if for all $\varepsilon>0$, there is some $R>0$ such that $|u(x)-c|<\varepsilon$ whenever $|x|>R$ and $x\in P^k$. 
			\end{definition}
			
			Recalling our convention from Definition~\ref{goog-rep}, we thus require that the precise representative of $u$ has the limit $c$. Let $M \subset \R^n$ be any $(n-k)$-dimensional submanifold of $\R^n$. At each $x\in M$ the tangent space $T_x M$ is a vector space of dimension $(n-k)$. Then $x+ (T_x M)^\perp$ is the unique $k$-dimensional affine normal subspace of $M$ inside $\R^n$ at $x$. We denote it by $P(x,M)$. Consider the collection $\mathcal{P}M:=\{P(x,M)\}_{x\in M}$. 
			\begin{definition}\label{a.e.limit}
				We say that a property holds for $\mathcal{H}^{n-k}$-almost every $k$-plane perpendicular to a $(n-k)$-dimensional submanifold $M \subset \R^n$, if there is some subset $E\subset M$ with $\mathcal{H}^{n-k}(E)=0$ and the property holds for any $P(x,M) \in \mathcal{P}M$ whenever $x\in M\setminus E$. 
			\end{definition} 
			
			\begin{remark}
				We define a canonical measure in $\mathcal{P}M$ by the push forward of the $(n-k)$-dimensional Hausdorff measure on $M$ via the bijection $x\mapsto P(x,M)$. With a slight abuse of notation, we continue to denote this newly defined measure by $\mathcal{H}^{n-k}$.
			\end{remark}

			\begin{definition}\label{ps-thin}
				Let $n\geq2$, $0<s<1$ and $0<p<\infty$. A subset $E\in\R^n$ is said to be \textit{$(p,s)-$ thin at infinity} if 
				\begin{equation}
					\lim_{m\to\infty}\mathcal{H}^{\lambda}_{\infty}(E\cap B(0,m)^c)=0\text{           for any $\lambda>n-sp$}.
				\end{equation}
			\end{definition}

			\begin{lemma}\cite[Lemma 3.1]{KlKo}\label{lm-Klein-original}
				Let $0< p<+\infty$. If $\{b_l\}_{l\in\N}$ is a sequence with $b_l\geq0$ and such that $\sum_{l\in\N}b_l<\infty$, then there exists a non-increasing sequence $\{a_l\}_{l\in\N}$ with $a_l>0$
				and $\lim_{l\to +\infty}a_l=0$, such that
				\begin{equation}
					\sum_{l\in\N}b_l a_l^{-p}<+\infty.
			\end{equation}\end{lemma}
			
			We will use the following inequalities in the sequel:
			\subsection*{Inequality}If $a_i\geq0$ for all $i$ and $0<p\leq1$, then
			\begin{equation} \label{ineq-p<1}
				\sum_{i\in\Z}a_i\leq \left(\sum_{i\in\Z}a_i^{p}\right)^{\frac{1}{p}}.
			\end{equation}

			\subsection*{Dyadic Decomposition}\label{1.2}
			Let $Q_0\subset\R^n$ be a cube. The dyadic decomposition $\mathcal{D}(Q_0)$ of $Q_0$ is defined as 
			\begin{equation}
				\mathcal{D}(Q_0)=\bigcup_{j=0}^\infty\mathcal{D}_j(Q_0),
			\end{equation} 
			where each $\mathcal{D}_j(Q_0)$ consists of $2^{jn}$ pairwise disjoint half open cubes $Q$ with edge length $\ell(Q)=2^{-j}\ell(Q_0)$ such that 
			\begin{equation}
				Q_0=\bigcup_{Q\in\mathcal{D_j}(Q_0)}Q
			\end{equation} 
			for every $j\in\N_0$. For a more detailed discussion about dyadic decomposition, we refer to \cite{Javier}.
			We continue with a useful lemma.
			\begin{lemma}\cite[Lemma 8.6]{KiLe}\label{lm-antti}
				Let $Q_0\subset\R^n$ be a cube. Assume that $u\in L^1(Q_0)$ is a non-negative function and let 
				\begin{equation}
					E_t=\left\{x\in Q_0\mid\sup_{x\in Q}\intav_{Q}|u(y)|dy>t\right\}
				\end{equation}
				for $t>0$. For every $t\geq u_{Q_0}$, there exists a collection $\mathcal{D}_t$ of pairwise disjoint dyadic cubes $Q\in\mathcal{D}(Q_0)$ such that
				\begin{equation}
					E_t=\bigcup_{Q\in\mathcal{D}_t}Q
				\end{equation}
				and \begin{equation}
					t<\intav_{Q}u(x)dx\leq2^nt 
				\end{equation}for every $Q\in\mathcal{D}_t$.
			\end{lemma}

			\bigskip	
			
			\section{Counterexamples}\label{counter examples}
			\begin{proof}[\textbf{Proof of Proposition~\ref{u-c for no c}}]
				Write $\lambda=\frac{n}{p}$. Define a function $u:\R^n\to\R$ by setting
				\begin{equation}
					u(x) = \begin{cases}
						1, \qquad &\text{if } |x|\leq1,\\
						|x|^{-\lambda}, \qquad &\text{if } |x|>1.
					\end{cases}
				\end{equation}
				Then
				\begin{equation}
					|\nabla u(x)| = \begin{cases}
						0, \qquad &\text{if } |x|\leq1,\\
						\lambda|x|^{-\lambda-1}, \qquad &\text{if } |x|>1.
					\end{cases}
				\end{equation}

				We will show that $u\in\dot{W}_q^{s,p}(\R^n)$. 
				By the fundamental theorem of calculus, we have 
				\begin{equation}\label{Case_x in R^n}
					\begin{split}
						\int_{ \R^n}\left(\int_{ B(0,2^{-1})} \frac{|u(x+h)-u(x)|^q}{|h|^{n+sq}}dh\right)^{\frac{p}{q}}dx
						&= \int_{ \R^n}\left(\int_{ B(0,2^{-1})} \frac{\left|\int_{0}^1 \nabla u (x+rh)\cdot h dr \right |^q}{|h|^{n+sq}}dh\right)^{\frac{p}{q}}dx\\
						&\leq \int_{ \R^n}(\sup_{y\in B(0,2^{-1})}|\nabla u (x+y)|)^p\left(\int_{ B(0,1)} \frac{1}{|h|^{n+sq-q}}dh\right)^{\frac{p}{q}}dx\\
						&\leq  C(n)\left(\frac{1}{q-sq}\right)^\frac{p}{q}\int_{ \R^n}(\sup_{z\in B(x,2^{-1})}|\nabla u (z)|)^pdx.
					\end{split}
				\end{equation}
				Thus we get
				\begin{equation}
					\begin{split}
						\int_{ \R^n}&\left(\int_{ B(0,2^{-1})} \frac{|u(x+h)-u(x)|^q}{|h|^{n+sq}}dh\right)^{\frac{p}{q}}dx\\
						& \qquad\leq  C(n,p,q,s)\left(\int_{B(0,1)}(\sup_{z\in B(x,2^{-1})}|\nabla u (z)|)^pdx + \int_{ B(0,1)^c}(\sup_{z\in B(x,2^{-1})}|\nabla u (z)|)^pdx\right)\\
						&\qquad\leq C(n,p,q,s,\lambda)\int_{ B(0,1)^c}\frac{1}{|x|^{p(\lambda+1)}}dx\\
						&\qquad\leq C(n,p,q,s,\lambda)\int_{1}^\infty \frac{t^{n-1}dt}{t^{(\lambda+1)p}}
						< \infty,
					\end{split}
				\end{equation}
				since $\lambda+1>\frac{n}{p}$.
				We are left to estimate
				\begin{equation}\label{mod x_eq1}
					\begin{split}
						\int_{ \R^n}\left(\int_{ B(0,2^{-1})^c} \frac{|u(x+h)-u(x)|^q}{|h|^{n+sq}}dh\right)^{\frac{p}{q}}dx
						&\leq \int_{B(0,2)}\left(\int_{B(0,2^{-1})^c} \frac{|u(x+h)-u(x)|^q}{|h|^{n+sq}}dh\right)^{\frac{p}{q}}dx\\
						&\qquad+\int_{ B(0,2)^c}\left(\int_{ B(0,2^{-1})^c} \frac{|u(x+h)-u(x)|^q}{|h|^{n+sq}}dh\right)^{\frac{p}{q}}dx.
					\end{split}
				\end{equation}
				First, let us deal with the first part of the right-hand side of the above inequality. We have
				\begin{equation}
					\begin{split}
						\int_{ B(0,2)}\left(\int_{ B(0,2^{-1})^c} \frac{|u(x+h)-u(x)|^q}{|h|^{n+sq}}dh\right)^{\frac{p}{q}}dx&\leq
						\int_{B(0,2)}\left(\sum_{j=0}^\infty \int_{ B(0,2^{j}) \setminus B(0,2^{j-1})} \frac{1}{|h|^{n+sq}}dh\right)^\frac{p}{q}dx\\
						&\leq C(n,p,q,s) \int_{ B(0,2)}\left(\sum_{j=0}^\infty  2^{-jsq} \right)^\frac{p}{q} dx<\infty.
					\end{split}
				\end{equation}
				We are left to show that the second part of the right-hand side of \eqref{mod x_eq1} is finite. We split the integral into three parts:
				\begin{equation}
					\begin{split}
						\int_{ B(0,2)^c}&\left(\int_{ B(0,2^{-1})^c} \frac{|u(x+h)-u(x)|^q}{|h|^{n+sq}}dh\right)^{\frac{p}{q}}dx\\
						&\leq\int_{ B(0,2)^c}\left(\int_{ B(0,2^{-1})^c\cap [\{|h|\geq 2|x|\}\setminus B(-x,1)}] \frac{|u(x+h)-u(x)|^q}{|h|^{n+sq}}dh\right)^{\frac{p}{q}}dx\\
						&\qquad\qquad+\int_{ B(0,2)^c}\left(\int_{ B(0,2^{-1})^c\cap [\{|h|<2|x|\}\setminus B(-x,1)}] \frac{|u(x+h)-u(x)|^q}{|h|^{n+sq}}dh\right)^{\frac{p}{q}}dx\\
						&\qquad\qquad+\int_{ B(0,2)^c}\left(\int_{ B(-x,1)} \frac{|u(x+h)-u(x)|^q}{|h|^{n+sq}}dh\right)^{\frac{p}{q}}dx\\
						&:=I_1+I_2+I_3. 
					\end{split}
				\end{equation}
				For $I_1$, we have
				\begin{equation}
					\begin{split}
						I_1&=\int_{ B(0,2)^c}\left(\int_{ B(0,2^{-1})^c\cap [\{|h|\geq 2|x|\}\setminus B(-x,1)]} \frac{|u(x+h)-u(x)|^q}{|h|^{n+sq}}dh\right)^{\frac{p}{q}}dx\\
						&\leq \int_{ B(0,2)^c}\left(\int_{B(0,2^{-1})^c\cap \{|h|\geq 2|x|\}} \frac{|\frac{1}{|x+h|^\lambda}-\frac{1}{|x|^\lambda}|^q}{|h|^{n+sq}}dh\right)^{\frac{p}{q}}dx\\
						&\leq\int_{ B(0,2)^c}\left(\int_{\{|h|\geq 2|x|\}} \frac{1}{|x|^{q\lambda}|h|^{n+sq}}dh\right)^{\frac{p}{q}}dx\\
						&\leq\int_{ B(0,2)^c}\frac{1}{|x|^{p\lambda}}\left(\int_{ 2|x|}^\infty \frac{t^{n-1}}{t^{n+sq}}dt\right)^{\frac{p}{q}}dx\\
						&\leq C(s,p,q)\int_{x\in B(0,2)^c}\frac{1}{|x|^{p\lambda+sp}}<\infty
					\end{split}
				\end{equation}
				since $|x+h|\geq ||h|-|x||\geq |h|-|x|\geq |x|$ and $\lambda+s>\frac{n}{p}$. Now for $I_3$ we have 
				\begin{equation}
					\begin{split}
						I_3&=\int_{ B(0,2)^c}\left(\int_{ B(-x,1)} \frac{|u(x+h)-u(x)|^q}{|h|^{n+sq}}dh\right)^{\frac{p}{q}}dx\\ 
						&=\int_{ B(0,2)^c}\left(\int_{ B(-x,1)} \frac{|1-u(x)|^q}{|h|^{n+sq}}dh\right)^{\frac{p}{q}}dx\\
						& \leq\int_{ B(0,2)^c}\left(\int_{B(-x,1)}\frac{1}{|h|^{n+sq}}dh\right)^{\frac{p}{q}}dx\\
						&\leq\int_{ B(0,2)^c}\left(\int_{ B(-x,1)}\frac{1}{(|x|-1)^{n+sq}}dh\right)^{\frac{p}{q}}dx\\
						&\leq C(n,p,q)\int_{ B(0,2)^c}\frac{1}{(|x|-1)^{p(n+sq)q^{-1}}}dx<\infty
					\end{split}
				\end{equation}
				since $q<\frac{np}{n-sp}$ gives $p(n+sq)q^{-1}>n$. For $I_2$, we have
				\begin{equation}
					\begin{split}
						I_2&=\int_{ B(0,2)^c}\left(\int_{ B(0,2^{-1})^c\cap [\{|h|<2|x| \}\setminus B(-x,1)]} \frac{|u(x+h)-u(x)|^q}{|h|^{n+sq}}dh\right)^{\frac{p}{q}}dx\\
						&\leq \int_{B(0,2)^c}\left(\int_{ B(0,2^{-1})^c\cap [\{B(-x,3|x|)\}\setminus B(-x,1)]}\frac{|u(x+h)-u(x)|^q}{|h|^{n+sq}}dh\right)^{\frac{p}{q}}dx\\
						&\leq \int_{B(0,2)^c}\left(\int_{ B(0,2^{-1})^c\cap [B(-x,\frac{|x|}{2})\setminus B(-x,1)]} \frac{|u(x+h)-u(x)|^q}{|h|^{n+sq}}dh\right)^{\frac{p}{q}}dx\\
						&+\int_{ B(0,2)^c}\left(\int_{ B(0,2^{-1})^c\cap [\{B(-x,3|x|)\setminus B(-x,\frac{|x|}{2})\}\setminus B(-x,1)]} \frac{|u(x+h)-u(x)|^q}{|h|^{n+sq}}dh\right)^{\frac{p}{q}}dx\\
						&:=I_{2,1}+I_{2,2}.
					\end{split}
				\end{equation}
				Above we have $h\in B(-x,3|x|)$ since $|h|<2|x|$ gives $|x+h|\leq|x|+|h|\leq 3|x|$. Now, we first evaluate $I_{2,1}$ as
				\begin{equation}
					\begin{split}
						I_{2,1}&=\int_{ B(0,2)^c}\left(\int_{ B(0,2^{-1})^c\cap [B(-x,\frac{|x|}{2})\setminus B(-x,1)]} \frac{|u(x+h)-u(x)|^q}{|h|^{n+sq}}dh\right)^{\frac{p}{q}}dx\\
						&\leq C(p)\int_{ B(0,2)^c}\left(\int_{  B(-x,\frac{|x|}{2})\setminus B(-x,1)} \frac{1}{|x+h|^{q\lambda}|h|^{n+sq}}dh\right)^{\frac{p}{q}}dx\\
						&\leq C(p,s,q,n)\int_{B(0,2)^c}\left(\int_{B(0,\frac{|x|}{2})\setminus B(0,1)} \frac{1}{|y|^{q\lambda}|x|^{n+sq}}dy\right)^{\frac{p}{q}}dx\end{split}\end{equation}
				since  $h\in  B(-x,\frac{|x|}{2})\setminus B(-x,1)$ gives $|h|>\frac{|x|}{2}$. Hence we obtain that
				\begin{equation}
					I_{2,1}\leq C(q,p,n,s) \int_{B(0,2)^c} \frac{1}{|x|^{sp+p\lambda}}dx<\infty,
				\end{equation}
				since $\lambda+s>\frac{n}{p}$ and  $|x|\geq 2$ gives $\left(\frac{|x|}{2}\right)^{n-q\lambda}\geq 1$. Lastly for $I_{2,2}$ we have
				\begin{equation}
					\begin{split}
						\int_{ B(0,2)^c}&\left(\int_{B(0,2^{-1})^c\cap [\{B(-x,3|x|)\setminus B(-x,\frac{|x|}{2})\}\setminus B(-x,1)]} \frac{|u(x+h)-u(x)|^q}{|h|^{n+sq}}dh\right)^{\frac{p}{q}}dx\\
						&\leq C(p)\int_{B(0,2)^c}\left(\int_{B(0,2^{-1})^c\cap [\{B(-x,3|x|)\setminus B(-x,\frac{|x|}{2})\}\setminus B(-x,1)]} \frac{1}{|x+h|^{q\lambda}|h|^{n+sq}}dh\right)^{\frac{p}{q}}dx\\
						&\leq C(p) \int_{ B(0,2)^c}\left(\int_{ B(0,2^{-1})^c\cap [\{B(-x,3|x|)\setminus B(-x,\frac{|x|}{2})\}\setminus B(-x,1)]} \frac{1}{|x|^{q\lambda}|h|^{n+sq}}dh\right)^{\frac{p}{q}}dx\\
						&\leq C(p)\int_{B(0,2)^c}\left(\int^{3|x|}_\frac{|x|}{2} \frac{t^{n-1-n-sq}}{|x|^{q\lambda}}dt\right)^{\frac{p}{q}}dx\\
						&\leq C(p,s,q,n)\int_{B(0,2)^c}\frac{1}{|x|^{p(\lambda+s)}} dx<\infty,
					\end{split}
				\end{equation}
				since $\lambda+s>\frac{n}{p}$ and $h\in B(-x,3|x|)\setminus B(-x,\frac{|x|}{2})$ gives $|x+h|>\frac{|x|}{2}$.
				Hence $u\in \dot{W}_q^{s,p}(\R^n)$ and clearly $u-c\notin L^p(\R^n)$ for any $p>0$ since otherwise $u$ must converge to $c$ which is not true. Also for $c=0$, $u\notin L^p(\R^n)$.
			\end{proof}	

				\begin{proof}[\textbf{Proof of Proposition~\ref{sp geq n lemma}}]
					Define the function $u:\R^n\to \R$ as
					\begin{equation}
						u(x)=\begin{cases}
							\log\log|x|, & \text{   if $|x|>e$},\\
							0, & \text{  if $|x|\leq e$}.
						\end{cases}
					\end{equation}
					Clearly, $\lim_{t\to \infty}u\circ \gamma (t)$ does not exist for any curve $\gamma:[0,\infty)\to\R^n$ with $\lim_{t\to\infty}\gamma(t)=\infty$.
					Note that
					\begin{equation}\label{nabla eqn}
						|\nabla u(x)| = \begin{cases}
							\frac{1}{|x|\log |x|}, \qquad &\text{if } |x|>e,\\
							0, \qquad &\text{if } |x|\leq e.
						\end{cases}
					\end{equation}
					As at the beginning of the proof of Proposition~\ref{u-c for no c}, the fundamental theorem of calculus gives  
					\begin{equation}\label{Case_x in R^n}
						\begin{split}
							\int_{ \R^n}\left(\int_{ B(0,1)} \frac{|u(x+h)-u(x)|^q}{|h|^{n+sq}}dh\right)^{\frac{p}{q}}dx
							&\leq  \int_{ \R^n}(\sup_{y\in B(0,1)}|\nabla u (x+y)|)^pdx\left(\int_{ B(0,1)}  \frac{1}{|h|^{n+sq-q}}dh\right)^{\frac{p}{q}}\\
							&\leq C(n,p,q,s) \int_{ \R^n}(\sup_{z\in B(x,1)}|\nabla u (z)|)^pdx.
						\end{split}
					\end{equation}
					Hence we get
					\begin{equation}
						\begin{split}
							\int_{ \R^n}&\left(\int_{B(0,1)} \frac{|u(x+h)-u(x)|^q}{|h|^{n+sq}}dh\right)^{\frac{p}{q}}dx\\
							& \qquad\leq  C(n,p,q,s)\left(\int_{B(0,e+1)}(\sup_{z\in B(x,1)}|\nabla u (z)|)^pdx + \int_{B(0,e+1)^c}(\sup_{z\in B(x,1)}|\nabla u (z)|)^pdx\right)\\
							&\qquad\leq C(n,p,q,s)\left(\frac{1}{e}+\int_{ B(0,e+1)^c}\frac{1}{(|x|-1)^p\log^p(|x|-1)}dx\right)\\
							&\qquad\leq C(n,p,q,s)\left( \frac{1}{e}+\int_{e+1}^\infty \frac{t^{n-p-1}dt}{(\log t)^p}\right)
							< \infty,
						\end{split}
					\end{equation}
					since $sp\geq n$ gives $p>n>1$. Next, 
					\begin{equation}\label{Case |x|<e}
						\begin{split}
							\int_{B(0,e)}&\left(\int_{h\in \R^n\setminus B(0,1)} \frac{|u(x+h)-u(x)|^q}{|h|^{n+sq}}dh \right)^\frac{p}{q}dx\\
							&= \int_{B(0,e)}\left(\sum_{j=1}^\infty \int_{ B(0,2^{j}) \setminus B(0,2^{j-1})} \frac{|u(x+h)|^q}{|h|^{n+sq}}dh\right)^\frac{p}{q}dx\\
							&\leq \int_{B(0,e)}\left(\sum_{j=1}^\infty \int_{ B(0,2^{j}) \setminus B(0,2^{j-1})} 2^{-jsq} \left| \log\log (2^{j}+|x|)\right|^q 2^{-jn}dh\right)^\frac{p}{q}dx\\
							&\leq C(n) \int_{B(0,e)}\left(\sum_{j=1}^\infty  2^{-jsq} \left| \log\log (2^{j}+|x|)\right|^q\right)^\frac{p}{q} dx.
						\end{split}
					\end{equation}
					Now, since $\log\log$ is an increasing function, we have by \eqref{Case |x|<e}
					\begin{equation}
						\begin{split}
							\int_{B(0,e)}\left(\int_{ \R^n\setminus B(0,1)} \frac{|u(x+h)-u(x)|^q}{|h|^{n+sq}}dh \right)^\frac{p}{q}dx
							&\leq \left(\sum_{j=1}^\infty  2^{-jsq} \left| \log\log (2^{j}+e)\right|^q\right)^\frac{p}{q}\\
							&\leq \left(\sum_{j=1}^\infty  2^{-jsq} \left| \log\log 2^{j+3}\right|^q\right)^\frac{p}{q} \\
							&\leq \left(\sum_{j=1}^\infty 2^{-jsq}\log^q 2(j+3)\right)^\frac{p}{q}<\infty.
						\end{split}
					\end{equation}
					
					For the last case, we have
                    \begin{equation}\label{p<q}
						\begin{split}
							\int_{B(0,e)^c}&\left(\int_{ \R^n\setminus B(0,1)} \frac{|u(x+h)-u(x)|^q}{|h|^{n+sq}}dh\right)^{\frac{p}{q}}dx\\
							&\leq C(p,q)\left( \int_{B(0,e)^c}\left(\int_{ B(0,|x|) \setminus B(0,1)}
                            \frac{|u(x+h)-u(x)|^q}{|h|^{n+sq}}dh\right)^{\frac{p}{q}}dx\right.\\
                            &\qquad\qquad+\left.\int_{B(0,e)^c}\left(\int_{ \R^n\setminus B(0,|x|) }
                            \frac{|u(x+h)-u(x)|^q}{|h|^{n+sq}}dh\right)^{\frac{p}{q}}dx\right)\\
                            &:= C(p,q)(I_1+I_2)
\end{split}\end{equation}
Let us first estimate $I_1$:
\begin{equation}
    \begin{split}
I_1&=\int_{B(0,e)^c}\left(\int_{ B(0,|x|) \setminus B(0,1)}
                            \frac{|u(x+h)-u(x)|^q}{|h|^{n+sq}}dh\right)^{\frac{p}{q}}dx\\
    &=   \int_{B(0,e)^c}    \left(\int_{ B(0,|x|) \setminus B(0,1)}
                           \left| \log\frac{\log(|x|+|h|)}{\log|x|}\right|^q\frac{dh}{|h|^{n+sq}}\right)^{\frac{p}{q}}dx \\
                            &=   \int_{B(0,e)^c}    \left(\int_{ B(0,|x|) \setminus B(0,1)}
                           \left| \log\frac{\log(|x|+|h|)-\log|x|+\log|x|}{\log|x|}\right|^q\frac{dh}{|h|^{n+sq}}\right)^{\frac{p}{q}}dx\\
                            &\leq   \int_{B(0,e)^c}    \left(\int_{ B(0,|x|) \setminus B(0,1)}
                           \left| \frac{\log\left(\frac{|x|+|h|}{|x|}\right)}{\log|x|}\right|^q\frac{dh}{|h|^{n+sq}}\right)^{\frac{p}{q}}dx\\
                           &\leq   \int_{B(0,e)^c}    \left(\int_{ B(0,|x|) \setminus B(0,1)}
                            \frac{|h|^{q-sq-n}}{|x|^q\log^q|x|}dh\right)^{\frac{p}{q}}dx\\
                            &\leq C(n,p,q)  \int_{B(0,e)^c}    \left(\int_{ 1}^{|x|}
                            \frac{t^{q-sq-1}}{|x|^q\log^q|x|}dt\right)^{\frac{p}{q}}dx\\
                             &\leq  C(n,p,q)  \int_{B(0,e)^c}    \frac{dx}{|x|^{sp}\log^p|x|}<\infty
                          \end{split}
\end{equation}
since $q-sq>0$ and $sp\geq n$. For $I_2$ we have
\begin{equation}
    \begin{split}
I_2&=\int_{B(0,e)^c}\left(\int_{\R^n\setminus B(0,|x|)}
                            \frac{|u(x+h)-u(x)|^q}{|h|^{n+sq}}dh\right)^{\frac{p}{q}}dx\\
                          &\leq C(n,p,q)   \int_{B(0,e)^c}    \left(\int_{|x|}^\infty
                           t^{-sq-1}\left| \log\frac{\log(|x|+t)}{\log|x|}\right|^qdt\right)^{\frac{p}{q}}dx \\
                             &\leq  C(n,p,q)  \int_{B(0,e)^c}    \left(\sum_{l=1}^\infty\int_{2^{l-1}|x|}^{2^l|x|}
                           t^{-sq-1}\left| \log\frac{\log(|x|+t)}{\log|x|}\right|^qdt\right)^{\frac{p}{q}}dx \\
                             &\leq C(n,s,p,q)   \int_{B(0,e)^c}    \left(\sum_{l=1}^\infty
                           (2^l|x|)^{-sq}\left| \log\frac{\log(|x|+2^l|x|)}{\log|x|}\right|^q\right)^{\frac{p}{q}}dx \\
 &C(n,s,p,q) \leq   \int_{B(0,e)^c}    \left(\sum_{l=1}^\infty
                           (2^l|x|)^{-sq}\left| \frac{\log2^{l+1}}{\log|x|}\right|^q\right)^{\frac{p}{q}}dx \\
                            &C(n,s,p,q) \leq   \int_{B(0,e)^c}\frac{dx}{|x|^{sp}\log^p|x|}    \left(\sum_{l=1}^\infty
                           l^q2^{-lsq}\right)^{\frac{p}{q}} <\infty
 \end{split}
\end{equation}
since $sp\geq n$ and $\sum_{l=1}^\infty
                           l^q2^{-lsq}<\infty.$ Therefore, \eqref{p<q} is finite, and we conclude that $u\in \dot{W}^{s,p}_q(\R^n)$.

			\end{proof}

			\begin{proof}[\textbf{Proof of Proposition~\ref{sp>n lemma}}]
				Take a function $v\in C_c^{\infty}(\R^n)$ such that
				\begin{equation}
					v(x)=\begin{cases}
						1, & \text{   on $S^{n-1}$},\\
						0, & \text{  on $B(0,2^{-1})\cup B(0,2)^c$}
					\end{cases}
				\end{equation}
				and 
				\begin{equation}
					|\nabla v| \leq 2.    
				\end{equation}
				Then $|\nabla v|=0$ on $B(0,2^{-1})\cup B(0,2)^c$. We first show that $v\in \dot{W}^{s,p}_q(\R^n)$. We will only deal with the case $q<1$, since for $q\geq1$, $C_c^\infty(\R^n)\subseteq W^{s,p}_q(\R^n)$ when $sp\geq n$ \cite[Remark~3.8]{PrEe}. First note that again
				
				\begin{equation}
					\begin{split}
						\int_{\R^n}\left(\int_{ B(0,2)} \frac{|v(x+h)-v(x)|^q}{|h|^{n+sq}}dh\right)^\frac{p}{q}dx
						&\leq \int_{ \R^n}(\max_{y\in B(0,2)}|\nabla v (x+y)|)^pdx\left(\int_{ B(0,1)} \frac{1}{|h|^{n+sq-q}}dh\right)^{\frac{p}{q}}\\
						&\leq C(n,p,q,s) \int_{ B(0,4)}(\sup_{z\in B(x,2)}|\nabla v (z)|)^pdx<\infty,
					\end{split}
				\end{equation}
				since $|\nabla v| \leq 2$. Also

				\begin{equation}
					\begin{split}
						\int_{ B(0,2)}\left(\int_{B(0,2)^c} \frac{|v(x+h)-v(x)|^q}{|h|^{n+sq}}dh\right)^\frac{p}{q}dx&=\int_{ B(0,2)}\left(\int_{ B(0,2)^c} \frac{|v(x)|^q}{|h|^{n+sq}}dh\right)^\frac{p}{q}dx\\
						&\leq\int_{B(0,2)}|v(x)|^pdx\left(\int_{ B(0,2)^c} \frac{1}{|h|^{n+sq}}dh\right)^\frac{p}{q}<\infty.
				\end{split}\end{equation}
				and for $q<1$ we have
				\begin{equation}
					\begin{split}
						\int_{B(0,2)^c}\left(\int_{ B(0,2)^c} \frac{|v(x+h)-v(x)|^q}{|h|^{n+sq}}dh\right)^\frac{p}{q}dx&=\int_{ B(0,2)^c}\left(\sum_{j=1}^{\infty}\int_{B(0,2^{j+1})\setminus B(0,2^j)}\frac{|v(x+h)|^q}{2^{j(n+sq)}}dh\right)^\frac{p}{q}dx\\
						&=\int_{ B(0,2)^c}\left(\sum_{j=1}^{\infty}2^{-jsq}\intav_{B(0,2^{j+1})\setminus B(0,2^j)}|v(x+h)|^qdh\right)^\frac{p}{q}dx\\
						&\leq \int_{x\in B(0,2)^c}\left(\sum_{j=1}^{\infty}2^{-jsq}\left|\intav_{B(x,2^{j+1})}{v}(y)dy\right|^q\right)^\frac{p}{q}dx
					\end{split}
				\end{equation}		
				Now	by the strong type estimate for the Hardy-Littlewood maximal operator, we get that
				\begin{equation}
					\begin{split}
						\int_{B(0,2)^c}\left(\int_{ B(0,2)^c} \frac{|v(x+h)-v(x)|^q}{|h|^{n+sq}}dh\right)^\frac{p}{q}dx&  \leq \int_{B(0,2)^c}\left(\sum_{j=1}^{\infty}2^{-jsq}|M{v}(x)|^q\right)^\frac{p}{q}dx\\
						&\leq \int_{ B(0,2)^c}|M{v}(x)|^pdx\\
						&\leq \int_{ \R^n}|M{v}(x)|^pdx\\
						&\leq  \int_{ \R^n}|v(x)|^pdx<\infty.
					\end{split}
				\end{equation}
				
				Let us first consider the case $sp>n$. Define $u_j(x)=v(A^{-j}x)$ where
				$A=(C+1)^{\frac{2p}{sp-n}}$ and 
				$C>1$ is as in \eqref{quasinorm constant}. Define 		
				\begin{equation}
					u=\sum_{j=1}^\infty u_j(x).
				\end{equation}
				Then clearly, $\lim_{t\to \infty}u\circ \gamma (t)$ does not exist for any curve $\gamma:[0,\infty)\to\R^n$ with $\lim_{t\to\infty}\gamma(t)=\infty$. Now we will show that $u\in\dot{W}^{s,p}_q(\R^n)$. We have

				\begin{equation}\label{uj estimate}
					\begin{split}
						[u_j]_{W^{s,p}_q(\R^n)}^p&=\int_{\R^n}\left(\int_{\R^n}\frac{|u_j(x)-u_j(y)|^q}{|x-y|^{n+sq}}dx\right)^\frac{p}{q}dy\\
						&\leq \int_{\R^n}\left(\int_{\R^n}\frac{|(v(A^{-j}x)-v(A^{-j}y)|^q}{|x-y|^{n+sq}}dx\right)^\frac{p}{q}dy\\
						&\leq A^{j(n-sp)} \int_{\R^n}\left(\int_{\R^n}\frac{|v(x)-v(y)|^q}{|x-y|^{n+sq}} dx\right)^\frac{p}{q}dy.
					\end{split}
				\end{equation}
				Now, since $sp> n$,
				\begin{equation}
					\begin{split}
						[u]_{W^{s,p}_q(\R^n)}
						\leq \sum_{j=1}^\infty C^j [u_j]_{W^{s,p}_q(\R^n)}
						\leq \sum_{j=1}^\infty (C+1)^{j}[u_j]_{W^{s,p}_q(\R^n)}
						\leq  C(n,p,s)\sum_{j=1}^\infty (C+1)^{-j}
						<\infty.
					\end{split}
				\end{equation}
				For the case $sp=n$, see the remark below.
			\end{proof}
			\begin{remark}
				The function 
				\begin{equation}
					u(x)=\begin{cases}
						\log\log|x|, & \text{   if $|x|>e$},\\
						0, & \text{  if $|x|\leq e$}.
					\end{cases}
				\end{equation}
				belongs to $\dot{W}^{s,p}_q(\R^n)$ for $sp\geq n$ and satisfies $u(x)\to\infty$ as $x\to\infty$. We can also construct a bounded continuous function for the case $sp\geq n$ so that there is no curve going to infinity along which the function has a limit. We construct the wall function as follows. First, let us define
				$u_0:\R^n\to \R$ as
				\begin{equation}
					u_0(x)=\begin{cases}
						\log\log|x|-\log\log R_0e, & \text{   if $R_0e<|x|\leq R_1e$},\\
						\log\log R_1e-\log\log(R_0e+|x|-R_1e), & \text{   if $R_1e<|x|\leq (2R_1-R_0)e$},\\
						0, & \text{   otherwise}.
					\end{cases}
				\end{equation}
				Then $\mathrm{supp}(u)\subset B(0,2R_1e)\setminus B(0,R_0e)$. Now define for each $j\in\N_0$, $u_{j+1}(x)=u_0(R_j^{-1}x)$ where $R_j=(C+1)^{(C+1)^{2j}} $ with $C>1$ is as in \eqref{quasinorm constant}. Then
				$$[u_{j+1}]^p_{W^{s,p}_q(\R^n)}=R_j^{n-sp}[u_0]^p_{W^{s,p}_q(\R^n)}.$$
				Now take 
				\begin{equation}
					v_{j+1}(x)=\frac{u_{j+1}(x)}{\log\log eR_{j+2}}
				\end{equation}
				then 
				$$[v_{j+1}]^p_{W^{s,p}_q(\R^n)}=\frac{[u_{j+1}]^p_{W^{s,p}_q(\R^n)}}{(\log\log eR_{j+2})^p}=R_j^{n-sp}\frac{[u_0]^p_{W^{s,p}_q(\R^n)}}{(\log\log eR_{j+2})^p}.$$
				Therefore, for $v(x)=\sum_{j=0}^\infty v_{j+1}(x)$ we have $$[v]_{W^{s,p}_q(\R^n)}\leq\sum_{j=0}^\infty (C+1)^j[v_j]_{W^{s,p}_q(\R^n)}\leq \sum_{j=0}^\infty(C+1)^jR_j^{(\frac{n}{p}-s)}\frac{[u_0]_{W^{s,p}_q(\R^n)}}{(\log\log R_{j+2}e)}<\infty.$$
				Hence $v\in\dot{W}^{s,p}_q(\R^n)$.
			\end{remark}

			\begin{proof}[\textbf{Proof of Proposition~\ref{sp leq 1 lemma}}]
				Let $sp<1$. Fix a function $v\in C_c^{\infty}(\R^n)$ with
				\begin{equation}
					v(x)=\begin{cases}
						1, & \text{   on $1-\varepsilon\leq|x|\leq 1+\varepsilon$},\\
						0, & \text{  on $B(0,1-2\varepsilon)\cup B(0,1+2\varepsilon)^c$}
					\end{cases}
				\end{equation}
				and so that
				\begin{equation}
					|\nabla v| \leq \frac{2}{\varepsilon}.    
				\end{equation}
				
				Now we will show that $v\in\dot{W}^{s,p}_q(\R^n)$. First, as before
				\begin{equation}
					\begin{split}
						\int_{ \R^n}\left(\int_{B(0,4\varepsilon)} \frac{|v(x+h)-v(x)|^q}{|h|^{n+sq}}dh\right)^\frac{p}{q}dx
						&\leq \int_{ \R^n}(\sup_{y\in B(0,4\varepsilon)}|\nabla v (x+y)|)^pdx\left(\int_{ B(0,4\varepsilon)} \frac{1}{|h|^{n+sq-q}}dh\right)^{\frac{p}{q}}\\
						&\leq  \int_{ \R^n}(\sup_{z\in B(x,4\varepsilon)}|\nabla v (z)|)^pdx\left(\int_{0}^{4\varepsilon}t^{n-1-(n+sq-q)}dt\right)^{\frac{p}{q}}\\
						&\leq \frac{1}{\varepsilon^{sp-p}} \int_{B(0,1+6\varepsilon)\setminus B(0,1-6\varepsilon)}(\sup_{z\in B(x,2\varepsilon)}|\nabla v (z)|)^pdx
						\\& \leq C(n,s,p,q)\varepsilon^{1-sp}.
					\end{split}
				\end{equation}
				Next, we have 
				\begin{equation}
					\begin{split}
						\int_{  B(0,1+2\varepsilon)\setminus B(0,1-2\varepsilon)}&\left(\int_{B(0,4\varepsilon)^c} \frac{|v(x+h)-v(x)|^q}{|h|^{n+sq}}dh\right)^\frac{p}{q}dx\\
						&=\int_{ B(0,1+2\varepsilon)\setminus B(0,1-2\varepsilon))}\left(\int_{ B(0,4\varepsilon)^c} \frac{|v(x)|^q}{|h|^{n+sq}}dh\right)^\frac{p}{q}dx\\
						&\leq\int_{ B(0,1+2\varepsilon)\setminus B(0,1-2\varepsilon)}|v(x)|^pdx\left(\int_{ B(0,4\varepsilon)^c} \frac{1}{|h|^{n+sq}}dh\right)^\frac{p}{q}\leq  C(n,s,p,q)\varepsilon^{1-sp}.
				\end{split}\end{equation}
				
				Lastly, we have,
				\begin{equation}
					\begin{split}
						\int_{ (B(0,1+2\varepsilon)\setminus B(0,1-2\varepsilon))^c}&\left(\int_{ B(0,4\varepsilon)^c} \frac{|v(x+h)-v(x)|^q}{|h|^{n+sq}}dh\right)^\frac{p}{q}dx\\&=\int_{ (B(0,1+2\varepsilon)\setminus B(0,1-2\varepsilon))^c}\left(\int_{ B(0,4\varepsilon)^c}\frac{|v(x+h)|^q}{|h|^{n+sq}}dh\right)^\frac{p}{q}dx\\
						&\leq \int_{(B(0,1-2\varepsilon)\setminus B(0,\varepsilon))\cup B(0,1+2\varepsilon)^c}\left(\int_{ B(-x,1+2\varepsilon)\setminus B(-x,1-2\varepsilon)}\frac{|v(x+h)|^q}{|h|^{n+sq}}dh\right)^\frac{p}{q}dx\\
						&\qquad\qquad\qquad+ \int_{ B(0,\varepsilon)}\left(\int_{ B(-x,1+2\varepsilon)\setminus B(-x,1-2\varepsilon)}\frac{|v(x+h)|^q}{|h|^{n+sq}}dh\right)^\frac{p}{q}dx\\
						&:=I_1+I_2
				\end{split}\end{equation}
				since for any $x\in B(0,1-2\varepsilon)\cup B(0,1+2\varepsilon)^c$, $\mathrm{supp}(v(x+h))\subset B(-x,1+2\varepsilon)\setminus B(-x,1-2\varepsilon)$. Now for $I_1$, we estimate $h\in B(-x,1+2\varepsilon)\setminus B(-x,1-2\varepsilon)$ gives $|h|\approx|x|$. Therefore we have
				\begin{equation}
					\begin{split}
						\int_{(B(0,1-2\varepsilon)\setminus B(0,\varepsilon))\cup B(0,1+2\varepsilon)^c}&\left(\int_{ B(0,4\varepsilon)^c} \frac{|v(x+h)-v(x)|^q}{|h|^{n+sq}}dh\right)^\frac{p}{q}dx\\&\approx \int_{ (B(0,1-2\varepsilon)\setminus B(0,\varepsilon))\cup B(0,1+2\varepsilon)^c}\left(\int_{ B(-x,1+2\varepsilon)\setminus B(-x,1-2\varepsilon)}\frac{1}{|x|^{n+sq}}dh\right)^\frac{p}{q}dx\\
						&\approx C(n,p,q)\varepsilon^{\frac{p}{q}}\int_{ (B(0,1-2\varepsilon)\setminus B(0,\varepsilon))\cup B(0,1+2\varepsilon)^c}\frac{1}{|x|^{\frac{np}{q}+sp}}dx\\
						&\approx C(n,p,q,s)\varepsilon^{\frac{p}{q}},
					\end{split}
				\end{equation}
				since $q<\frac{np}{n-sp}$ gives $\frac{np}{q}+sp>n$, integral over $x$ is finite. Now for $I_2$ we have
				\begin{align}
					\int_{ B(0,\varepsilon)}&\left(\int_{ B(-x,1+2\varepsilon)\setminus B(-x,1-2\varepsilon)}\frac{|v(x+h)|^q}{|h|^{n+sq}}dh\right)^\frac{p}{q}dx\\
					&\leq \varepsilon^n \left(\int_{ B(-x,1+2\varepsilon)\setminus B(-x,1-2\varepsilon)}\frac{1}{|h|^{n+sq}}dh\right)^\frac{p}{q}dx\\
					&\leq C(n,s,p,q) \varepsilon^{n-sp}.
				\end{align}
				Thus we have $v\in\dot{W}^{s,p}_q(\R^n)$. Now let $A=C+1$ where $C>1$ is as in \eqref{quasinorm constant}. Define for each $j\in\N$, $u_j:\R^n\to\R$ as $u_j(x)=v(A^{-j}x)$ and take $\varepsilon=\varepsilon_j$ for $u_j$ where $\varepsilon_j=A^{-j\alpha}$ and $\alpha=2\frac{(n-sp+p)}{\max\{1-sp,pq^{-1},n-sp\}}$. Then $\mathrm{supp}(u_j)=B(0,A^j+2\varepsilon_j)\setminus B(0,A^j-2\varepsilon_j)$. Then  
				\begin{equation}
					\begin{split}
						[u_j]_{W^{s,p}_q(\R^n)}^p&=\int_{\R^n}\left(\int_{\R^n}\frac{|u_j(x)-u_j(y)|^q}{|x-y|^{n+sq}}dx\right)^\frac{p}{q}dy\\
						&\leq \int_{\R^n}\left(\int_{\R^n}\frac{|(v(A^{-j}x)-v(A^{-j}y)|^q}{|x-y|^{n+sq}}dx\right)^\frac{p}{q}dy\\
						&\leq A^{j(n-sp)} \int_{\R^n}\left(\int_{\R^n}\frac{|v(x)-v(y)|^q}{|x-y|^{n+sq}} dx\right)^\frac{p}{q}dy\\
						&\leq C(n,p,q,s)A^{j(n-sp)} (\varepsilon_j^{1-sp}+\varepsilon_j^{\frac{p}{q}}+\varepsilon_j^{n-sp}).
					\end{split}
				\end{equation}
				Define
				$$u(x)=\sum_{j=1}^\infty u_j(x).$$ Then
				\begin{equation}
					\begin{split}
						[u]_{W^{s,p}_q(\R^n)}
						\leq \sum_{j=1}^\infty C^j [u_j]_{W^{s,p}_q(\R^n)}
						\leq\sum_{j=1}^\infty (C+1)^{-j}<\infty.
					\end{split}
				\end{equation}     
				Clearly, $\lim_{t\to \infty}u\circ \gamma (t)$ does not exist for any curve $\gamma:[0,\infty)\to\R^n$ with $\lim_{t\to\infty}\gamma(t)=\infty$.

				\end{proof}	
				\begin{proof}[\textbf{Proof of Proposition~\ref{sp=1}}]
					Denote $A_j= B(0,2^{j+1})\setminus B(0,2^j)$ and let $S_j$ be the middle $(n-1)$-dimensional sphere in $A_j$. Since $\mathcal{H}^{n-1}(S_j)<\infty$, especially, $\mathcal{H}^{n-sp}(S_j)<\infty$ since $sp=1$. Therefore, by Lemma~\ref{haus_cap_AdDa} and Lemma~\ref{capacity comparison}, we have $\mathrm{C}_{s,p}(S_j)=0$ so $\mathrm{Cap}_{s,p,q}(S_j)=0$ as well. 
					Moreover, we have that
					$$ [u]_{W^{s,p}_q(\R^n)}
					\leq\sum_{j=1}^{\infty} C^j[u]_{W^{s,p}_q(A_j)}$$
					where $C>1$ is the constant obtained by the property of quasinorm.
					Thus we can choose a function  $u_j\in {W}_q^{s,p}(\R^n)$ for every $j\in\N$ such that $u_j\geq1$ on a neighbourhood of $S_j$, and $u_j=0$ for $\R^n\setminus A_j$ with $\|u_j\|_{W^{s,p}_q(\R^n)}\leq 2^{-j}$. Consider $u=\sum_{j=1}^{\infty}u_j$. Then 
					\begin{equation}
						\|u\|_{W^{s,p}_q(\R^n)}
						\leq\sum_{j=1}^{\infty}  \|u\|_{W^{s,p}_q(A_j)}
						\leq\sum_{j=1}^{\infty}2^{-j}<\infty.
					\end{equation}
					Clearly, $\lim_{t\to \infty}u\circ \gamma (t)$ does not exist for any curve $\gamma:[0,\infty)\to\R^n$ with $\lim_{t\to\infty}\gamma(t)=\infty$.
				\end{proof}

				\bigskip
				
				\section{Proof of Theorem~\ref{ps-thin theorem}}\label{main}

				\begin{lemma}\label{one length cube conv}
					Let $0<s<1$ and $0<p,q<\infty$ with $sp<n$. Let $\varepsilon>0$ be given. Then, for every $u\in\dot{W}^{s,p}_q(\R^n)$, there exists $M\in\N$ and a unique $K_u\in\R$ so that 
					\begin{equation}\label{uniform conv cube}
						|m_u(Q)-K_u|<\varepsilon
					\end{equation}
					whenever $Q\subset\R^n$ is a cube with $Q\cap B^n(0,M)=\emptyset$ and $\ell(Q)=1$.
				\end{lemma}
				\begin{proof}
					It suffices to establish the claim for cubes $Q$ with $\ell(Q)=1$	and  $Q\cap B^n(0,M)=\emptyset$. Recall that our cubes are required to have edges parallel to the coordinate axes. Fix $M\geq1$ and such a $Q$. Then there is $j\in\{1,2,\cdots,n\}$ so that $|x_j|> M$ holds for each $x=(x_1,x_2,\cdots,x_n)\in Q$. For simplicity, assume that $j=n$ and that $|x_n|>M$ for each $x\in Q$. Define $\tilde{Q}$ to be the orthogonal projection of $Q$ in $\R^{n-1}\times \{0\}$. Then $Q=\tilde{Q}\times [t,t+\ell(Q))$ for some $t>M$ by our convention that cubes are of the form $[a_1,b_1)\times \cdots \times [a_n,b_n)$. Define $Q_j=2^j\tilde{Q}\times (t+2^jl({Q}),t+2^{j+1}l({Q}))$ for $j\geq1$. Then $Q_0=Q$, the cubes $Q_j$ are pairwise disjoint and $Q_j\cap B^n(0,M)=\emptyset$ for all $j\in\N_0$, in fact $Q_j\cap B^n(0,t)=\emptyset$. 
					Now let $j>i$. Then we have by H\"{o}lder's inequality and Lemma \ref{l1} that
					\begin{align}\label{1.3.1}
						|m_u({Q}_j)-m_u({Q}_i)|
						&\leq\sum_{k=i}^{j-1}|m_u({Q}_{k+1})-m_u({Q}_k)|\\
						&\leq C(p,q)\sum_{k=i}^{j-1}\left(\intav_{{Q}_{k+1}}\left(\intav_{{Q}_{k}}|u(x)-u(y)|^qdx\right)^{\frac{p}{q}}dy\right)^\frac{1}{p}\\
						&\leq C(p,q)\sum_{k=i}^{j-1}\diam({Q}_k\cup {Q}_{k+1})^{\frac{n+sq}{q}}\left(\intav_{{Q}_{k+1}}\left(\intav_{{Q}_{k}}\frac{|u(x)-u(y)|^q}{|x-y|^{n+sq}}dx\right)^{\frac{p}{q}}dy\right)^\frac{1}{p}\\
						&\leq C(p,q,s,n)\sum_{k=i}^{\infty}\frac{2^{k\left(\frac{n+sq}{q}\right)}}{2^{kn\left(\frac{1}{p}+\frac{1}{q}\right)}} \left(\int_{{Q}_{k+1}}\left(\int_{{Q}_{k}}\frac{|u(x)-u(y)|^q}{|x-y|^{n+sq}}dx\right)^{\frac{p}{q}}dy\right)^\frac{1}{p}\\
						&\leq  C(p,q,s,n)\left(\int_{B^n(0,t)^c}\left(\int_{B^n(0,t)^c}\frac{|u(x)-u(y)|^q}{|x-y|^{n+sq}}dx\right)^{\frac{p}{q}}dy\right)^\frac{1}{p}\sum_{k=i}^{\infty}2^{k\left(s-\frac{n}{p}\right)}.
					\end{align}
					Let $\varepsilon>0$. By $sp<n$ and inequality \eqref{1.3.1} we conclude that $$|m_u({Q}_j)-m_u({Q}_i)|<\varepsilon$$
					when $i<j$ and $i$ is sufficiently large. Hence $\{m_u({Q}_j)\}_{j\in\N}$ is a Cauchy sequence and therefore there exists a constant $K_u^{({Q})}\in\R$ such that $\lim_{j\to\infty}m_u({Q}_j)=K_u^{({Q})}$. We show that this limit is independent of $Q$. Consider two distinct cubes $Q, R$ of edge length 1 in $\R^n$ and define a sequence $\{{R}_j\}_{j\in\N_0}$ in a similar manner to $\{{Q}_j\}_{j\in\N_0}$ so that  
					$$\lim_{j\to\infty}m_u({Q}_j)=K_u^{({Q})}, \hspace{5pt}\lim_{j\to\infty}m_u({R}_j)=K_u^{({R})}.$$
					Then by Lemma \ref{l1} we have
					\begin{align}\label{1.3.2}
						|m_u({Q}_j)-m_u({R}_j)|&\leq C(p,q)\left(\intav_{{Q}_{j}}\left(\intav_{{R}_j}|u(x)-u(y)|^qdx\right)^{\frac{p}{q}}dy\right)^\frac{1}{p}\\
						&\leq C(p,q)\frac{(\mathrm{diam}(Q,R)2^j)^{\left(\frac{n+sq}{q}\right)}}{2^{jn\left(\frac{1}{p}+\frac{1}{q}\right)}}\left(\int_{{Q}_{j}}\left(\int_{{R}_j}\frac{|u(x)-u(y)|^q}{|x-y|^{n+sq}}dx\right)^{\frac{p}{q}}dy\right)^\frac{1}{p}\\
						&\leq C(p,q,s,n)2^{j\left(s-\frac{n}{p}\right)}\left(\int_{\R^n}\left(\int_{\R^n}\frac{|u(x)-u(y)|^q}{|x-y|^{n+sq}}dx\right)^{\frac{p}{q}}dy\right)^\frac{1}{p}.\end{align}
					Thus by inequality \eqref{1.3.2}, since $sp<n$ we conclude that
					$$\lim_{j\to\infty}|m_u({Q}_j)-m_u({R}_j)|=0.$$
					Hence, the limit is unique and we will denote it $K_u$. Now we have by \eqref{1.3.1},
					\begin{align}\label{limit-cube-p>1}
						|m_u(Q)-K_u|&=\lim_{j\to\infty}|m_u(Q)-m_u({Q}_j)|\\
						&=\lim_{j\to\infty}|m_u({Q}_0)-m_u({Q}_j)|\\
						&\leq C(p,q,s,n)\left(\int_{B^n(0,t)^c}\left(\int_{B^n(0,t)^c}\frac{|u(x)-u(y)|^q}{|x-y|^{n+sq}}dx\right)^{\frac{p}{q}}dy\right)^\frac{1}{p}
					\end{align}
					where $t$ is as chosen at the beginning of the proof. Now, since $u\in\dot{W}^{s,p}_q(\R^n)$, for any $\varepsilon>0$, there exists $M\in\N$ so that, for $t\geq M$, 
					\begin{equation} \label{ball complement estimate}
						C(p,q,s,n) \left(\int_{B^n(0,t)^c}\left(\int_{B^n(0,t)^c}\frac{|u(x)-u(y)|^q}{|x-y|^{n+sq}}dx\right)^{\frac{p}{q}}dy\right)^\frac{1}{p}<\varepsilon
					\end{equation}
					where $C$ is the constant in \eqref{limit-cube-p>1}. We combine \eqref{limit-cube-p>1} and \eqref{ball complement estimate} to conclude that
					\begin{equation}
						|m_u(Q)-K_u|<\varepsilon.
					\end{equation}
					
					%
				\end{proof}
				\begin{remark}
					The claim of the above lemma actually holds for any cube $Q\subset B^n(0,M)^c\subset \R^n$ not necessarily having edges parallel to coordinate axes. For our purposes, it is enough to consider cubes having edges parallel to coordinate axes.
				\end{remark}

				\begin{lemma}\label{l3}
					Let $0<s<1$ and $0<p,q<\infty$ with $sp<n$. Let $Q$ be a cube in $\R^n$ with $\ell(Q)=1$. For any $\lambda>n-sp$ there is $C=C(n,s,p,q,\lambda)>0$ such that for all $u\in \dot{W}^{s,p}_q(\R^n)$, we have
					\begin{equation}\label{haus-cube}
						\sup_{t>0}t^p\mathcal{H}^{\lambda}_{\infty}(\{x\in Q \ | \ |u(x)-m_u(Q)|>t\})\leq 
						C\int_{Q}\left(\int_{Q}\frac{|u(z)-u(y)|^q}{|z-y|^{n+sq}}dz\right)^{\frac{p}{q}}dy.
					\end{equation}
				\end{lemma}

				\begin{proof}
					Fix $t>0$. Define 
					\begin{equation}E_t:=\{x\in Q \ | \ |u(x)-m_u(Q)|>t\}.\end{equation}
					Let $x\in E_t$ and $\{A_i\}_{i\in\N}$ be a sequence of cubes obtained by dyadically decomposing $Q$ so that  $A_1=Q, A_i\subset Q$ for all $i\in \N,$ $ l(A_i)=2^{-i+1}l(Q)$ and $x\in A_i$ for all $i$.
					We use Lemma~\ref{l1} on $A_l$ and $A_{l+1}$ so that we have
					\begin{align}
						t<|u(x)-m_u(Q)|&=\limsup_{i\to\infty}|m_u(A_i)-m_u(Q)|\\
						&=\limsup_{i\to\infty}|m_u(A_i)-m_u(A_1)|\\
						&\leq\sum_{l=1}^{\infty}|m_u(A_{l+1})-m_u(A_l)|\\
						&\leq C(p,q)\sum_{l=1}^{\infty}\left(\intav_{{A}_{l+1}}\left(\intav_{{A}_{l}}|u(z)-u(y)|^qdz\right)^{\frac{p}{q}}dy\right)^\frac{1}{p}\\
						&\leq C(p,q)\sum_{l=1}^{\infty}\left(\intav_{A_{l+1}}\left(\intav_{A_{l}}\diam(A_l\cup A_{l+1})^{n+sq}\frac{|u(z)-u(y)|^q}{|z-y|^{n+sq}}dz\right)^{\frac{p}{q}}dy\right)^\frac{1}{p}\\
						&\leq C(p,q,s,n)\sum_{l=1}^{\infty}\frac{2^{-l\left(\frac{n+sq}{q}\right)}}{2^{-ln\left(\frac{1}{p}+\frac{1}{q}\right)}}\left(\int_{A_{l+1}}\left(\int_{A_{l}}\frac{|u(z)-u(y)|^q}{|z-y|^{n+sq}}dz\right)^{\frac{p}{q}}dy\right)^\frac{1}{p}\\
						&\leq C(p,q,s,n)\sum_{l=1}^{\infty}2^{-l\left(s-\frac{n}{p}\right)}\left(\int_{A_{l}}\left(\int_{A_{l}}\frac{|u(z)-u(y)|^q}{|z-y|^{n+sq}}dz\right)^{\frac{p}{q}}dy\right)^\frac{1}{p}.\label{1.2.2}
					\end{align}
					Set $\varepsilon=\lambda-n+sp$. Then by inequality \eqref{1.2.2} we have that
					$$t\sum_{l\geq1}2^{-l\varepsilon}\leq C(p,q,s,n) \sum_{l\geq 1}2^{-l\left(s-\frac{n}{p}\right)}\left(\int_{A_{l}}\left(\int_{A_{l}}\frac{|u(z)-u(y)|^q}{|z-y|^{n+sq}}dz\right)^{\frac{p}{q}}dy\right)^\frac{1}{p}.$$
					Hence, for every $x$ in $E_t$,
					we can find $l_x\in\N$ such that $A_{l_x}\subset Q$, $x\in A_{l_x}$ and
					\begin{equation}t2^{-l_x\varepsilon}\leq C(p,q,s,n) 2^{-l_x\left(s-\frac{n}{p}\right)}\left(\int_{A_{l_x}}\left(\int_{A_{l_x}}\frac{|u(z)-u(y)|^q}{|z-y|^{n+sq}}dz\right)^{\frac{p}{q}}dy\right)^\frac{1}{p}.
					\end{equation}
					This gives
					\begin{equation}\label{1.2.3}t^p2^{-l_x\lambda}\leq C(p,q,s,n)\int_{A_{l_x}}\left(\int_{A_{l_x}}\frac{|u(z)-u(y)|^q}{|z-y|^{n+sq}}dz\right)^{\frac{p}{q}}dy.\end{equation}
					Now by considering Lemma \ref{lm-antti}, one finds a countable family $\mathcal{D}_t$ of pairwise disjoint $\{A_{l_x}\}_{x\in E_t}$ such that the cubes in $\mathcal{D}_t$ 
					cover $E_t$.
					Hence we may write $E_t\subset\cup_{A_{l_x}\in\mathcal{D}_t}A_{l_x}=\cup_{l=1}^{\infty}A_l$. Thus, by \eqref{1.2.3}, we conclude that
					\begin{align}
						\mathcal{H}^{\lambda}_{\infty}(E_t)
						&\leq\sum_{l\geq1}\diam(A_l)^{\lambda}\\
						&\leq C(s,p,q,n,\lambda)\sum_{l\geq1}(2^{-l+1}\ell(Q))^{\lambda}\\
						&\leq C(s,p,q,n,\lambda)\frac{1}{t^p}\sum_{l\geq 1}\int_{A_l}\left(\int_{A_l}\frac{|u(z)-u(y)|^q}{|z-y|^{n+sq}}dz\right)^{\frac{p}{q}}dy\\
						&\leq C(s,p,q,n,\lambda)\frac{1}{t^p}\int_{Q}\left(\int_{Q}\frac{|u(z)-u(y)|^q}{|z-y|^{n+sq}}dz\right)^{\frac{p}{q}}dy.
					\end{align}
				\end{proof}

				\begin{proof}[Proof of Theorem~\ref{ps-thin theorem}]
					Let $\{Q_j\}_{j\in \N}$ be the enumeration of all cubes in $\R^n$ having edge-length one, and with vertices in $\Z^n$. Define a sequence $\{a_{Q_j}\}_j$, where $a_{Q_j}:=a_j$, obtained by Lemma~\ref{lm-Klein-original}, is such that
					\begin{equation}\label{new result aj estimate}
						a_j>0,\ \lim_{j\to\infty}a_j=0,\ \sum_{j\in\N}\frac{1}{a_j^p}\int_{Q_j}\left(\int_{\R^n}\frac{|u(x)-u(y)|^q}{|x-y|^{n+sq}}dx\right)^{\frac{p}{q}}dy<\infty.   
					\end{equation}Define, for all $j\in\N$,
					$$E_j:=\{x\in Q_j\mid |u(x)-m_u(Q_j)|>a_j\}.
					$$
					Set $E:=\bigcup_{j\in\N}E_j$. Now
					\begin{align}\label{haus_a_j_spherical_thin}
						\lim_{m\to\infty}\mathcal{H}^{\lambda}_{\infty}(E\cap B(0,m)^c)&=
						\lim_{m\to\infty} \mathcal{H}^{\lambda}_{\infty}(\cup_{j\in\N}E_j \cap B(0,m)^c).
					\end{align}
					Since $Q_{j}\to\infty$ as $j\to\infty$,  for every $m>0$ we find a $i_m\in\N$ so that $Q_{j}\subset B(0,m)^c$ whenever $j>i_m$. Therefore for $\lambda>n-sp$ we have by Lemma~\ref{l3}, \eqref{haus_a_j_spherical_thin} and \eqref{new result aj estimate} that
					\begin{align}
						\lim_{m\to\infty}\mathcal{H}^{\lambda}_{\infty}(E\cap B(0,m)^c)&=\lim_{m\to\infty} \mathcal{H}^{\lambda}_{\infty}(\cup_{j>i_m}E_{j})\\
						&\leq \limsup_{m\to\infty}\sum_{j>i_m}\mathcal{H}^{\lambda}_{\infty}(E_{j})\\
						&\leq \limsup_{m\to\infty}\sum_{j>i_m}\frac{1}{a_{j}^p}\int_{Q_{j}}\left(\int_{\R^n}\frac{|u(x)-u(y)|^q}{|x-y|^{n+sq}}dx\right)^{\frac{p}{q}}dy=0.
					\end{align}
					Let $x\in \R^n\setminus E$. Choosing $Q$ such that $x\in Q$, we get
					\begin{align}\label{conv}
						|u(x)-K_u|\leq |u(x)-m_u(Q)|+|m_u(Q)-K_u|<a_{Q}+|m_u(Q)-K_u|. 
					\end{align}
					Let $\varepsilon>0$ be given. Then by \eqref{new result aj estimate} and Lemma~\ref{one length cube conv}, we choose $M\in\N$ sufficiently large so that the above gives $|u(x)-K_u|<\varepsilon$, whenever $x\in B(0, M)^c\setminus E$.
				\end{proof}
				
				\begin{remark}\label{bad_set and a_j} 
					We observe from the above proof that $u\in\dot{W}^{s,p}_q(\R^n)$ will approach $K_u$ at infinity along every curve in $\R^n$ which tends to infinity and hits the set $E$ only finitely many times. Also, since for a non-increasing sequence $a_j\to0$, the set $E= ~\bigcup_{j\in\N}\{x \in Q_j \ | \ |u(x)-m_u(Q_j)|>a_j\}$ has a very small $\mathcal{H}^{n-1}$ measure at infinity when $1<sp<n$, we have plenty of curves along which $u$ has a limit. 
				\end{remark}	
				
				\begin{lemma}\label{median convg Lp}
					Let $0<r<\infty$ and suppose $u-c\in L^r(\R^n)$ for a given $c\in\R$. Then, for every $\varepsilon>0$, there exists $M\in\N$ so that 
					\begin{equation}\label{uniform conv cube}
						|m_u(Q)-c|<\varepsilon
					\end{equation}
					whenever $Q$ is a cube with $Q\cap B^n(0,M)=\emptyset$ and $\ell(Q)=1$.
				\end{lemma}
				\begin{proof}
					Fix $c\in\R$ with $u-c\in L^r(\R^n)$. Let $Q$ be a cube with $\ell(Q)=1$. Then by Lemma~\ref{med_another lemma} we get that 
					\begin{equation}
						|m_u(Q)-c|\leq |u(x)-c|
					\end{equation}
					for a set $\tilde{Q}\subset Q$ with $|\tilde{Q}|\geq \frac{|Q|}{2}$. Therefore we deduce that
					\begin{equation}\label{med and c}
						\int_{\tilde Q}|m_u(Q)-c|^r dx\leq \int_{\tilde Q}|u(x)-c|^rdx
						\leq \int_{Q}|u(x)-c|^rdx.
					\end{equation}
					Let $\varepsilon>0$ be given. Now, since $u-c\in L^r(\R^n)$, then for $\delta=\frac{\varepsilon^r}{2}$ there is $M\in\N$ so that 
					\begin{equation}\label{small norm}
						\int_{B(0,M)^c}|u-c|^r<\delta.
					\end{equation}
					Suppose $Q\subset B^n(0,M)^c$. Then, we obtain from \eqref{med and c} and \eqref{small norm} that \begin{equation}
						|m_u(Q)-c|\leq \left(2\int_{B^n(0,M)^c}|u(x)-c|^rdx\right)^\frac{1}{r}<\varepsilon.
					\end{equation}

				\end{proof}
				\begin{proof}[Proof of Corollary~\ref{Lp convg 0}]
					Lemma~\ref{median convg Lp} combined with Lemma~\ref{l3} gives the claim of Corollary~\ref{Lp convg 0}.  
				\end{proof}
				Now we will prove Corollary~\ref{only constant function}.
				\begin{proof}[Proof of Corollary~\ref{only constant function}]
					Let $u\in\dot{W}^{s,p}_q(\R^n)$ with $0<s<1$, $0<p<\infty$, $q\geq\frac{np}{n-sp}$ and $sp<n$. Let $K_{u^*}$ and $E$ be as in Theorem~\ref{ps-thin theorem}. Contrary to the claim, let us suppose that $u(x)\neq K_{u^*}$ on a set of positive volume. Then there exists a bounded set $A\subset\R^n$ such that $|A|>0$ and $|u-K_{u^*}|\geq\varepsilon$ in $A$.
					Choose $j_1,j_2\in\N $ such that $A\subset B(0,2^{j_1})$ and $|u-K_{u^*}|<\frac{\varepsilon}{2}$ in $B(0,2^{j_2})^c\setminus E$ respectively. Here the set $E$ is as in Theorem~\ref{ps-thin theorem}. Define $A_j=B(0,2^{j+1})\setminus B(0,2^{j})$. Since $\mathcal{H}^{\lambda}(E)<\infty$, especially $|E|<\infty$, we may choose a $j_3\in\N$ such that $|A_j|\geq2|E|$ for all $j\geq j_3$. Let $j_0=\max\{j_1,j_2,j_3\}$. Then we have 
					\begin{align}
						\int_{\R^n}\left(\int_{\R^n}\frac{|u(x)-u(y)|^q}{|x-y|^{n+sq}}dy\right)^\frac{p}{q}dx&\geq\int_{B(0,2^{j_0})^c\setminus E}\left(\int_{A}\frac{|u(x)-u(y)|^q}{|x-y|^{n+sq}}dy\right)^\frac{p}{q}dx\\
						&\geq\frac{\varepsilon^p}{2^p}\int_{B(0,2^{j_0})^c\setminus E}\left(\int_{A}\frac{1}{|x-y|^{n+sq}}dy\right)^\frac{p}{q}dx\\
					\end{align}
					Now since $|y|\leq |x|$ gives $|x-y|\leq|x|+|y|\leq2|x|$, we have
					\begin{align}
						\int_{\R^n}\left(\int_{\R^n}\frac{|u(x)-u(y)|^q}{|x-y|^{n+sq}}dy\right)^\frac{p}{q}dx&
						\geq\frac{\varepsilon^p}{2^{(n+sq+q)\frac{p}{q}}}\int_{B(0,2^{j_0})^c\setminus E}\left(\int_{A}\frac{1}{|x|^{n+sq}}dy\right)^\frac{p}{q}dx\\
						&\geq \frac{\varepsilon^p|A|^{\frac{p}{q}}}{2^{(n+sq+q)\frac{p}{q}}}\sum_{j=j_0}^\infty\int_{A_j\setminus E}\frac{1}{|x|^{(n+sq)\frac{p}{q}}}dx\\
						&\geq \frac{\varepsilon^p|A|^{\frac{p}{q}}}{2^{(n+sq+q)\frac{p}{q}}}\sum_{j=j_0}^\infty\frac{|A_j\setminus E|}{2^{j(n+sq+q)\frac{p}{q}}}\\
						&\geq \frac{\varepsilon^p|A|^{\frac{p}{q}}}{2^{(n+sq+q)\frac{p}{q}+1}}\sum_{j=j_0}^\infty \frac{|A_j|}{2^{j(n+sq)\frac{p}{q}}}\\
						&\geq \frac{\varepsilon^p|A|^{\frac{p}{q}}}{2^{(n+sq+q)\frac{p}{q}+1}}\sum_{j=j_0}^\infty 2^{j(n-(n+sq)\frac{p}{q})}=\infty
					\end{align}
					since $q\geq \frac{np}{n-sp}$. This is not possible since $u$ was assumed to be in the space $\dot{W}^{s,p}_q(\R^n)$. Hence the claim follows. 
				\end{proof}

				\begin{lemma}\label{haus-cap}
					Let $0<s<1$, $1<p<\infty$, $1\leq q<\frac{np}{n-sp}$ and $sp<n$. If $E\subset\R^n$ is a compact set with $\mathcal{H}^{n-sp}(E)<\infty$ and $\varepsilon>0$, there is $u\in W^{s,p}_q(\R^n)$ such that $u\geq1$ on $E$ and $\|u\|_{W^{s,p}_q(\R^n)}<\varepsilon$.
				\end{lemma}
				\begin{proof}
					Let $E$ be a compact set in $\R^n$ so that $\mathcal{H}^{n-sp}(E)<\infty$. Then by Lemma~\ref{haus_cap_AdDa}, we have $\mathrm{C}_{s,p}(E)=0$ and hence $\mathrm{Cap}_{s,p,q}(E)=0$ by Lemma~\ref{capacity comparison}. Thus for any given $\varepsilon>0$ there exist a $u_{\varepsilon}\in W^{s,p}_q(\R^n)$ such that $u_{\varepsilon}\geq1$ on a neighbourhood of $E$ and $\|u_{\varepsilon}\|_{W^{s,p}_q(\R^n)}<\varepsilon$.
				\end{proof}
				
				Next, we show that the Hausdorff estimate in Lemma~\ref{l3} is optimal, in the sense that we cannot control $\mathcal{H}^\lambda_\infty(E_t)$ for $\lambda=n-sp$. This is formally stated in Remark~\ref{rem.sharpness} below.
				\begin{proposition}\label{counterex-sp-1,n}
					Let  $0<s<1$, $1<p<\infty$, $1\leq q<\frac{np}{n-sp}$ and $sp<n$. Let $Q$ be a cube of edge length $\mathrm{1}$ in $\R^n$. Then for every $\varepsilon>0$ there exists a function $u\in W^{s,p}_q(\R^n)$ so that $\|u\|_{W^{s,p}_q(\R^n)}<\varepsilon$ and  $\sup_{t>0}t^p\mathcal{H}^{n-sp}_{\infty}(E_{t})\geq C$ for some fixed constant $C>0$, where $E_{t}=\{x\in Q \mid |u(x)-m_u(Q)|>t\}$.   
				\end{proposition}
				\begin{proof}
					Let $E\subset \frac{1}{2}Q$ be compact and $(n-sp)$-Ahlfors regular with $\mathcal{H}^{n-sp}(E)\equiv\frac{1}{2^{n-sp}}$. For construction of such sets see \cite{Mattila}. Let $\psi\in C_c^{\infty}(Q)$ be such that $\psi=1$ on $\frac{1}{2}Q$ and $\mathrm{supp}(\psi)\subset \frac{2}{3}Q$. Then, $\mathcal{H}_{\infty}^{n-sp}(Q\setminus\text{supp}(\psi))\geq \mathcal{H}^{n-sp}_\infty(Q\setminus\frac{2}{3}Q)
					\geq \frac{1}{3}$. Let $\varepsilon>0$ be given. Then by Lemma~\ref{haus-cap} we find $u_{\varepsilon}\in W^{s,p}_q(\R^n)$ such that $u_{\varepsilon}\geq1$ on a neighbourhood of $E$ and $\|u_{\varepsilon}\|_{W^{s,p}_q(\R^n)}<\frac{\varepsilon}{\|\psi\|_{W^{s,p}_q(\R^n)}}$. Thus, for $v=u_{\varepsilon}\psi$, we have
					\begin{align}\label{4.2.1}
						\mathcal{H}^{n-sp}_{\infty}(\{x\in Q\mid v(x)=0\})&\geq\mathcal{H}^{n-sp}_{\infty}(Q\setminus\text{supp}(u_{\varepsilon}\psi))
						\geq \frac{1}{3}
					\end{align}
					and by \cite[Lemma~7.8]{KiLe} we have
					\begin{align}\label{4.2.2}
						\mathcal{H}^{n-sp}_{\infty}(\{x\in Q\mid v(x)=1\})&\geq\mathcal{H}^{n-sp}_{\infty}(E)\geq C(C_A,n,s,p)\frac{1}{2^{n-sp}}.
					\end{align}
					where $C_1$ is the Ahlfors-regularity constant of $E$. Therefore, by the above two estimates, we have
					\begin{align}
						\frac{1}{3^p}\mathcal{H}^{n-sp}_{\infty}(\{x\in Q\mid |v(x)-m_v(Q)|>\frac{1}{3}\})&\geq C
					\end{align}
					where $C=\min\{C(C_A,n,s,p)\frac{1}{2^{n-sp}},C(p)\frac{1}{3}\}$. Hence the claim follows.
				\end{proof}
				\begin{remark}\label{rem.sharpness} If we let $\lambda=n-sp$ in Lemma~\ref{l3}, Proposition \ref{counterex-sp-1,n} implies that there exists $v\in\dot{W}^{s,p}_q(\R^n)$ such that the left-hand side of \eqref{haus-cube} remains bounded away from zero while the right-hand side becomes arbitrarily small.
				\end{remark}
				\begin{remark}
					The above lemma does not cover the case $q<1<p$ as the capacity results for $q<1$ are not known. For $p\leq1$ we have $sp<1$, and then the case $q<1$ is contained in Proposition~\ref{sp leq 1 lemma}.
				\end{remark}
				
				\bigskip
				
				\section{The case of radial limits}\label{radial section}

				
				\begin{proof}[\textbf{Proof of Corollary~\ref{rad coro}}]
					Without loss of generality, we assume $$
					\mathbb{S}^{n-k}= \{x\in \R^{n-k}\times \{0\}^k \mid |x|=1\}.
					$$

					\textbf{Step-1:} Construction of \textit{the set $F$ of bad parameters}.\\
					
					Consider $E\subset \R^n$, mentioned in Remark~\ref{bad_set and a_j}. Set $E_j:=E\setminus B(0,j)$ for all $j\in \N$.

					Consider the orthogonal projection
					$$
					\pi: \R^n\to \mathbb{S}^{n-k}
					$$
					given by 
					$$
					\pi((x_1,\cdots,x_n))= \frac{(x_1,\cdots,x_{n-k+1},0,\cdots,0)}{|(x_1,\cdots,x_{n-k+1},0,\cdots,0)|}.
					$$
					Note that $\pi(P_z^k)=\{z\}$. Define
					$$
					F:=\bigcap_{j=1}^\infty \pi(E_j).
					$$
					
					\smallskip
					
					\textbf{Step-2:} Analysis of the complement of $F$ inside $\mathbb{S}^{n-k}$.\\
					
					Let $z\in \mathbb{S}^{n-k} \setminus F$ and $\varepsilon>0$ be chosen arbitrarily. Since $a_i\to 0$, let $i_0\in \N$ be such that $i>i_0$ implies $a_i<\varepsilon$. Now, we have
					$$
					z\in \bigcup_{j=1}^\infty \left(\mathbb{S}^{n-k} \setminus \pi(E_j)\right).
					$$
					That is, there is some $j_0\in \N$ such that $z\in \mathbb{S}^{n-k} \setminus \pi(E_{j_0})$, consequently, for every $j\geq j_0$, 
					$$
					\left(P_z^k\setminus B(0,j)\right)\cap E=P_z^k\cap E_{j}=\emptyset.
					$$
					
					\begin{itemize}
						\item Choose $M_1>0$ such that $B(z,M_1) \supset B(0,j_0)$.
						
						\item By Lemma~\ref{one length cube conv}, choose $M_2>0$ such that  and $Q\subset \R^n \setminus B(z,M_2)$ with $\ell(Q)=1$ implies
						\begin{equation}\label{eq1}
							|m_u(Q)-K_u|<\frac{\varepsilon}{2}.
						\end{equation}
						
						\item Choose $M_3>0$ such that $Q_i\subset \R^n\setminus B(0,M_3)$ implies $i>i_0$. Consequently, $Q_i\subset \R^n\setminus B(0,M_3)$ implies $a_i<\varepsilon$.
						
						\item Choose $M=2 \max\{ M_1,M_2,M_3 \}$.
					\end{itemize}
					Choose $\xi\in P_z^k\setminus B(z,M) \subset P_z^k\setminus B(0,j_0)\subset \R^n\setminus E$. Then, for some $i(\xi)\in \N$, $\xi\in Q_{i(\xi)}\subset \R^n\setminus B(0,M_2)$. Also note that $a_{i(\xi)}<\frac{\varepsilon}{2}$. Hence, we have
					$$
					|u(\xi)- K_u|
					\leq |m_u(Q_{i(\xi)})-K_u|+|u(\xi)-m_u(Q_{i(\xi)})|<
					|m_u(Q_{i(\xi)})- K_u|+a_{i(\xi)}<\varepsilon.
					$$

					\textbf{Step-3:} The Hausdorff content estimate of $E$.\\

					Let $\lambda>n-sp$. Theorem~\ref{ps-thin theorem} together with Definition~\ref{ps-thin}, gives
					\begin{align}
						\mathcal{H}^{\lambda}_{\infty}(F)
						\leq\limsup_{j\to\infty}\mathcal{H}^{\lambda}_{\infty}(\pi(E_j))
						\leq \limsup_{j\to\infty} \mathcal{H}^{\lambda}_{\infty}(E_j)\leq
						\limsup_{j\to\infty} \mathcal{H}^{\lambda}_{\infty}(E\setminus B(0,j))
						=0.
					\end{align}
					Now, since $\lambda>n-sp$ and $sp>k$, we have $\mathcal{H}^{n-k}(F)=0$.
				\end{proof}	
				\bigskip
				\section{The case of vertical limits}\label{vertical section}


				\begin{proof}[\textbf{Proof of Corollary~\ref{ver coro}}]
					
					Without loss of generality let us assume that 
					$$
					V^{n-k}=\R^{n-k}\times\{0\}^k.
					$$
					\smallskip
					
					\textbf{Step-1:} Construction of \textit{the set $\mathcal{S}$ of bad parameters}.\\
					
					Consider $E\subset \R^n$, mentioned in Remark~\ref{bad_set and a_j}. Set $E_j:=E\setminus B(0,j)$ for all $j\in \N$.
					Consider the orthogonal projection 
					$$
					\pi: \R^n\to \R^{n-k}
					$$
					defined by $\pi((x_1,\cdots,x_n))= (x_1,\cdots,x_{n-k})$. Note that $\pi(P_z^k)=\{z\}$. Define
					$$
					\mathcal{S}:= \left(\bigcap_{j\geq1} \pi(E_{j})\right)\times\{0\}^k.
					$$
					\smallskip
					
					\textbf{Step-2:} Analysis of the complement of $\mathcal{S}$ inside $\R^{n-k}\times \{0\}^k$.\\
					
					Let, for some $x_1\in \R^{n-k}$,  $(x_1,0)\in (\R^{n-k}\times \{0\}^k) \setminus\mathcal{S}$.  Then
					$$
					x_1\in \R^{n-k}\setminus \bigcap_{j\geq1} \pi(E_{j})
					=\bigcup_{j\geq1}  \left(\R^{n-k}\setminus \pi(E_j)\right).
					$$
					That is, there is some $j_0\in \N$ such that for every $j\geq j_0$, we have $x_1\in \left(\R^{n-k}\setminus \pi(E_j)\right)$. In other words, there is some $j_0\in \N$ such that for every $j>j_0$, and every $x_2\in \R^k$, $x:=(x_1,x_2)\notin E_j$.

					Let $\varepsilon>0$ be chosen arbitrarily. 
					
					Since $a_j\to 0$, we may choose $j_1> j_0$ such that $a_j<\frac{\varepsilon}{2}$ for $j>j_1$. Choose $L>0$ be so large that $Q_j\cap \left(\R^n \setminus B((x_1,0),L)\right) \neq \emptyset$ implies $|m_u(Q_j)- K_u|<\frac{\varepsilon}{2}$ (follows from Lemma~\ref{one length cube conv}) and $j>j_1$.
					
					Choose an arbitrary point $(x_1,y)$ of $P_{x_1}^k$ with $|y|>L$. Assume $(x_1,y)\in Q_j$; so that $j>j_1>j_0$, consequently $(x_1,y)\notin E_j$. So, we have shown that $|y|>L$ and $(x_1,y)\in Q_j$ implies 
					$$
					|u(x_1,y)-m_u(Q_j)|<a_j.
					$$

					We use the previous estimates to conclude that
					$$
					|u(x_1,y)-K_u|
					\leq |m_u(Q_j)-K_u|+|u(x_1,y)-m_u(Q_j)|<
					|m_u(Q_j)-K_u|+a_j<\varepsilon
					$$
					whenever $|y|>L$. 
					\smallskip

					\textbf{Step-3:} The Hausdorff content estimate of $E$.\\

					Let $\lambda>n-sp$. Theorem~\ref{ps-thin theorem}, together with Definition~\ref{ps-thin}, gives
					\begin{align}
						\mathcal{H}^{\lambda}_{\infty}(\mathcal{S})\leq\limsup_{j\to\infty}\mathcal{H}^{\lambda}_{\infty}(\pi(E_{j}))
						\leq \limsup_{j\to\infty}\mathcal{H}^{\lambda}_{\infty}(E_j)=0.
					\end{align}
					Thus for $\lambda>n-sp$ and $sp>k$, we have $\mathcal{H}^{n-k}(\mathcal{S})=0$.
					
					
				\end{proof}

			\bibliographystyle{plain}
			\bibliography{1st.bib}
			
			\setlength{\parindent}{0pt}

		\end{document}